\documentclass[12pt,a4paper]{amsart}
\usepackage{amsthm}
\usepackage{amssymb}
\usepackage{amsmath}
\usepackage{enumitem}
\usepackage[english]{babel}
\usepackage{url}
\usepackage{tikz-cd}
\usepackage[margin=2.7cm]{geometry}
\usepackage[labelformat=simple,labelfont=bf]{subfig}

\theoremstyle{plain}
\newtheorem{theorem}{Theorem}[section]
\newtheorem{lemma}[theorem]{Lemma}
\newtheorem{proposition}[theorem]{Proposition}
\newtheorem{corollary}[theorem]{Corollary}

\theoremstyle{definition}
\newtheorem{defn}[theorem]{Definition}
\newtheorem{ex}[theorem]{Example}

\theoremstyle{remark}
\newtheorem{rem}[theorem]{Remark}

\newcommand{\cD}{{\mathcal D}}
\newcommand{\cF}{{\mathcal F}}
\newcommand{\cG}{{\mathcal G}}

\newcommand{\cN}{{\mathcal N}}
\newcommand{\cP}{{\mathcal P}}
\newcommand{\cT}{{\mathcal T}}
\newcommand{\cU}{{\mathcal U}}
\newcommand{\cX}{{\mathcal X}}
\newcommand{\cZ}{{\mathcal Z}}
\newcommand{\F}{{\mathbb{F}}}
\newcommand{\N}{{\mathbb{N}}}
\newcommand{\Z}{{\mathbb{Z}}}
\newcommand{\R}{{\mathbb{R}}}
\newcommand{\mvmap}{\rightrightarrows}
\newcommand{\diam}{\mathop{\mathrm{diam}}\nolimits}
\newcommand{\walk}{{\rightsquigarrow}}
\newcommand{\Int}{\mathop{\mathrm{int}}\nolimits} 
\newcommand{\cl}{\mathop{\mathrm{cl}}\nolimits}
\newcommand{\bdy}{\mathop{\mathrm{bdy}}\nolimits}
\newcommand{\Inv}{\mathop{\mathrm{Inv}}\nolimits}
\newcommand{\Con}{\mathop{\mathrm{Con}_*}\nolimits}
\newcommand{\sA}{{\mathsf{ A}}}
\newcommand{\sJ}{{\mathsf{ J}}}
\newcommand{\sK}{{\mathsf{ K}}}

\newcommand{\sN}{{\mathsf{ N}}}
\newcommand{\sP}{{\mathsf{ P}}}
\newcommand{\sAtt}{{\mathsf{ Att}}}

\newcommand{\sABlock}{{\mathsf{ ABlock}}}
\newcommand{\sInvset}{{\mathsf{ Invset}}}
\newcommand{\sMG}{{\mathsf{ MG}}}
\newcommand{\sMR}{{\mathsf{ MR}}}
\newcommand{\sMT}{{\mathsf{ MT}}}
\newcommand{\sMTile}{{\mathsf{ MTile}}}
\newcommand{\sSC}{{\mathsf{ SCC}}}
\newcommand{\sCon}{{\mathsf{Con}}}
\newcommand{\bzero}{{\mathbf{0}}}
\newcommand{\bone}{{\mathbf{1}}}

\newcommand{\setof}[1]{\left\{ {#1}\right\}}
\newcommand{\setdef}[2]{\left\{{#1}\,\left|\,{#2}\right.\right\}}
\newcommand{\btheta}{{\bar{\theta}}}

\title{Rigorously Characterizing Dynamics with Machine Learning}
\author[Marcio Gameiro, Brittany Gelb, and Konstantin Mischaikow]{Marcio Gameiro\textsuperscript{1}, Brittany Gelb\textsuperscript{1}, and
Konstantin~Mischaikow\textsuperscript{1}
}
\address{\textsuperscript{1} Department of Mathematics, Rutgers, The State University of New Jersey, 110 Frelinghusen Rd., Piscataway, NJ 08854-8019, USA}

\begin{document}

\begin{abstract}
The identification of dynamics from time series data is a problem of general interest.
It is well established that dynamics on the level of invariant sets, the primary objects of interest in the classical theory of dynamical systems, is not computable.
We recall a coarser characterization of dynamics based on order theory and algebraic topology and prove that this characterization can be identified using approximations.
\end{abstract}

\maketitle
\section{Introduction}

The goal of this paper is to present a framework that
allows for the use of machine learning to rigorously characterize
nonlinear dynamics generated by iterates of a continuous function $f$ defined on a compact subset $X$ of $\R^d$.
The fundamental assumption is that given sufficient data and given $\rho > 0$ it is possible to learn a function $G\colon \R^d\to \R^d$ that satisfies the following three properties.
\begin{description}
    \item[P1] $\sup_{x\in X}\| f(x) - G(x)\| < \rho$.
    \item[P2] Given $x\in X$ we can evaluate $G(x)$.
    \item[P3] We can determine explicit Lipschitz bounds on $G$.
\end{description}
Our strategy is to use $G$ to characterize the dynamics of $f$.

The  classical theory of dynamical systems focuses on understanding the existence and structure of \emph{invariant sets}, i.e., sets of the form $S\subset X$ such that $f(S) = S$. 
This leads to a rich mathematical theory but, as we argue below, a theory that is too rich from the perspective of deducing dynamics from data.

With regard to the goals of this paper, the fundamental question is: is there a $\rho$ such that $G$ (satisfying {\bf P1} - {\bf P3}) generates the same dynamics as $f$?
In the classical setting this is equivalent to asking: can we guarantee that $f$ and $G$ are \emph{conjugate},  i.e., does there exist a homeomorphism $h\colon X\to X$ such that
\[
f \circ h = h \circ G?
\]
The work of \cite{foreman:rudolph:weiss} proves that in general this is not possible via computations.
More precisely, regarding their work the authors state
\begin{quote}
    “This can be interpreted as saying that there is no method or protocol that involves a countable amount of information and countable number of steps that reliably distinguishes between nonisomorphic ergodic measure preserving transformations. We view this as a rigorous way of saying that the classification problem for ergodic measure preserving transformations is intractable.’’
\end{quote}

An alternative approach is to focus on individual invariant sets.
However,  as shown by the following example, even this is challenging.

\begin{ex}
\label{ex:bifurcation}
A trivial example that demonstrates the challenge of identifying  invariant sets based on approximations is as follows. 
Consider $f_\theta\colon \R \to \R$ defined by
\[
f_\theta(x) = \begin{cases}
    0 & \text{if $x\leq 1/2$}\\
    2 x - 1 & \text{if $1/2\leq x\leq (\theta +1)/2$} \\
    \theta & \text{if $x\geq (\theta +1)/2$} 
\end{cases}
\]
and shown in Figure~\ref{fig:ftheta}.
Observe that if $\theta <1$, then $f_\theta$ has one fixed point, $x=0$; if $\theta >1$, then $f_\theta$ has three fixed points, $x=0$, $x =1$, and $x= \theta$; and  $f_1$ has two fixed points, $x=0$ and $x=1$.

\begin{figure}[!htb]
\centering
\subfloat[]{\begin{tikzpicture}[scale=0.6]
\draw[->] (0,0) -- (4,0)  node[right]{$x$};
\draw[->] (0,0) -- (0,4)  node[above]{$y$};
\draw (0,0) -- (4,4)  node[above]{$y=x$};
\draw (1,-0.1) -- (1,0.1) node[below]{$\frac{1}{2}$};
\draw (2,-0.1) -- (2,0.1) node[below]{$1$};
\draw (-0.1,1.5) -- (0.1,1.5)  node[left]{$\theta$};
\draw[blue, thick] (0,0) -- (1,0) -- (1.75,1.5) -- (4,1.5) node[above]{$y=f_\theta(x)$};
\draw[blue, fill=blue] (0,0) circle (0.5ex);
\end{tikzpicture}}
\hspace{0.5cm}
\subfloat[]{\begin{tikzpicture}[scale=0.6]
\draw[->] (0,0) -- (4,0)  node[right]{$x$};
\draw[->] (0,0) -- (0,4)  node[above]{$y$};
\draw (0,0) -- (4,4)  node[above]{$y=x$};
\draw (1,-0.1) -- (1,0.1) node[below]{$\frac{1}{2}$};
\draw (2,-0.1) -- (2,0.1) node[below]{$1$};
\draw (-0.1,2) -- (0.1,2) node[left]{$1$};
\draw[blue, thick] (0,0) -- (1,0) -- (2,2) -- (4,2) node[above]{$y=f_1(x)$};
\draw[blue, fill=blue] (0,0) circle (0.5ex);
\draw[blue, fill=blue] (2,2) circle (0.5ex);
\end{tikzpicture}}
\hspace{0.5cm}
\subfloat[]{\begin{tikzpicture}[scale=0.6]
\draw[->] (0,0) -- (4,0)  node[right]{$x$};
\draw[->] (0,0) -- (0,4)  node[above]{$y$};
\draw (0,0) -- (4,4)  node[above]{$y=x$};
\draw (1,-0.1) -- (1,0.1) node[below]{$\frac{1}{2}$};
\draw (2,-0.1) -- (2,0.1) node[below]{$1$};
\draw (3,-0.1) -- (3,0.1) node[below]{$\theta$};
\draw (-0.1,2) -- (0.1,2) node[left]{$1$};
\draw (-0.1,3) -- (0.1,3) node[left]{$\theta$};
\draw[blue, thick] (0,0) -- (1,0) -- (2.5,3) -- (4,3) node[below]{$y=f_\theta(x)$};
\draw[blue, fill=blue] (0,0) circle (0.5ex);
\draw[blue, fill=blue] (2,2) circle (0.5ex);
\draw[blue, fill=blue] (3,3) circle (0.5ex);
\end{tikzpicture}}
\caption{(a) The function $f_\theta$ for $\theta<1$ that has a fixed point at $0$.
(b) The function $f_1$ that has fixed points at $0$ and  $1$.
(c) The function $f_\theta$ for $\theta > 1$ that has fixed points at $0$, $1$, and $\theta$.
}
\label{fig:ftheta}
\end{figure}

Returning to the goal of this paper, we are allowed to assume that the map we want to learn the dynamics of is $f_1\colon \R\to \R$, but we are not allowed to assume that we know the functional form of $f_1$.
We are also allowed to assume that we have a finite (but arbitrarily large) set of data
$\setof{ (x_k, f_1(x_k)) \mid x_k\in [-2,2]}$
that can be used to obtain an approximation $G\colon \R\to \R$ such that $\sup_{x\in [-2,2]}|f_1(x) - G(x)|<\rho$.
Observe that even for arbitrarily small $\rho >0$, we cannot guarantee that $G$ has exactly two fixed points, and therefore we cannot guarantee that we can learn the number of fixed points of $f_1$ from $G$.
\end{ex}

While Example~\ref{ex:bifurcation} may appear to be too trivial to be of interest, it is indicative of a more serious challenge.
Consider a continuous parameterized family of maps $f\colon X\times [0,1]^n \to X$, where for each $\lambda \in [0,1]^n$, $f_\lambda := f(\cdot,\lambda)$.
A parameter value $\lambda_0$ is a \emph{bifurcation point} if for any neighborhood $U$ of $\lambda_0$ there exists $\lambda\in U$ such that $f_\lambda$ is not conjugate to $f_{\lambda_0}$.
In particular, $\theta = 1$ is a bifurcation point in Example~\ref{ex:bifurcation}.
As is shown in \cite{newhouse,palis:takens} if $d\geq 2$ (recall that $X\subset \R^d$), then there are parameterized families of maps for which the bifurcation points contain a Cantor set of positive measure.
As a consequence, in general it is not possible to identify the dynamics of $f$ via an approximation.

We avoid the above mentioned challenges by employing a coarser framework for dynamics that we believe to be sufficiently refined for many applications.
Figure~\ref{fig:overview} presents an overview of our approach.

\begin{figure}[!htb]
    \centering
    \includegraphics[width=0.9\linewidth]{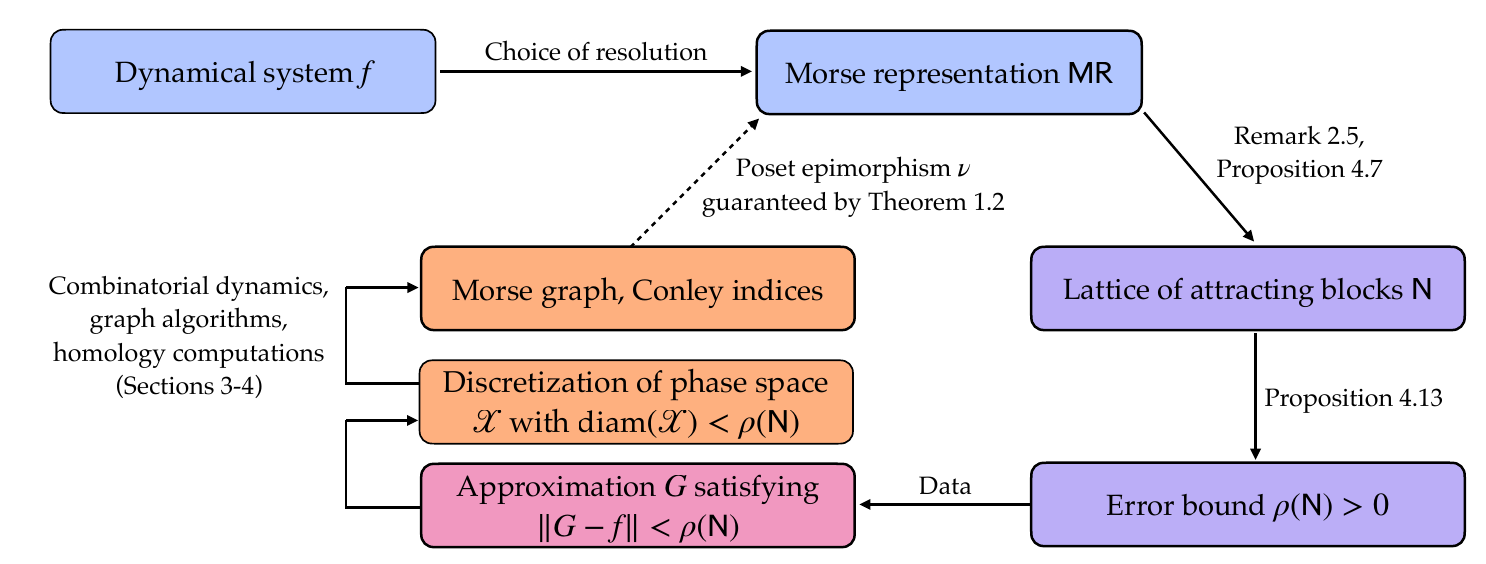}
    \caption{
    Overview of the proposed framework. 
    (Blue) Given a dynamical system $f$, a Morse representation $\sMR$ is equivalent to a choice of resolution of the dynamics. In applications $f$ and $\sMR$ are unknown.
    (Purple) Theory guarantees the existence of a lattice of attracting blocks $\sN$ that in turn determines an error bound $\rho$.  In applications $\sN$ and $\rho$ are unknown.
    (Pink) Given $\rho$, we assume the existence of an approximation $G$ that satisfies {\bf P1}. In applications, the approximation $G$ is derived from data generated by $f$.
    (Orange) Given $G$, we compute a combinatorial model $\cG$ that takes the form of a directed graph on a discretization of the phase space $\cX$. 
    Graph algorithms produce a Morse graph, which captures the gradient-like structure of the dynamics. Homology computations produce Conley indices, topological invariants that provide further information about the recurrent dynamics.
    (Dashed line) Theorem~\ref{thm:intro:main} guarantees that an order-preserving surjection $\nu$ from the Morse graph to the chosen Morse representation exists. As a consequence the computations provide correct information for any system satisfying the error bound.
    }
    \label{fig:overview}
\end{figure}

Since we do not expect that the typical reader is familiar with this framework and its associated machinery, a significant portion of this paper is spent introducing the necessary definitions and concepts.
In particular, in Section~\ref{sec:ConleyThy} we provide precise definitions of the terms used in the following theorem, which is our main result.

\begin{theorem}
\label{thm:intro:main}
Let $X\subset \R^d$ be a
finite union of closed, bounded, convex $d$-dimensional polytopes.
Let $f$ be a continuous function such that $f(X)\subset \Int(X)$  where $\Int$ denotes interior.
Let $(\sMR,\leq_{\sMR_f})$ be a Morse representation of $f$.
For each Morse set $M\in \sMR$, let $\sCon_*(M)$ denote the homology Conley index of $M$ under the dynamics induced by $f$.

Let $G\colon \R^{d}\to\R^{d}$ satisfy {\bf P1} - {\bf P3}.
If $\rho >0$ is sufficiently small, then using $G$ we can
\begin{enumerate}
    \item [(i)] Compute a Morse graph $(\sMG,\leq_\sMG)$ for which there exists a poset epimorphism, 
    \[
    \nu\colon \sMG \to \sMR, 
    \]
    and 
    \item [(ii)] Use $\nu$ to determine $\Con(M)$ for each $M\in\sMR$.
\end{enumerate}
\end{theorem}

An imprecise statement of Theorem~\ref{thm:intro:main} is as follows.
Choose a Morse representation $\sMR$ (see Definition~\ref{defn:MorseRepresentation}) for the function $f$.
In general $f$ has many (possibly infinitely many) Morse representations, thus the choice of $\sMR$ can be interpreted as a choice of resolution of the dynamics of $f$ that is of interest. 
The theorem implies that we can employ an approximation $G$ to (i) organize the global dynamics of $f$ in terms of its gradient-like behavior, and (ii) compute  homological invariants (Conley indices) from which we can obtain rigorous statements concerning the dynamics of $f$.

To provide slightly more detail, a Morse graph (see Definition~\ref{defn:MorseGraph}) is a partially ordered set (poset) whose elements identify potential recurrent dynamics and whose order structure indicates the asymptotic direction of nonrecurrent dynamics.
The Conley index $\sCon_*$ (see Definition~\ref{defn:ConleyIndex}) is a computable \cite{mischaikow:weibel} algebraic topological extension of the Morse index that we employ to identify nonlinear dynamics.
Most fundamentally, a non-trivial Conley index implies the existence of a non-empty invariant set, however the converse is \emph{not} true. 
The Conley index can also be used to prove the existence of equilibria \cite{srzednicki:85,mccord:88}, periodic orbits \cite{mccord:mischaikow:mrozek}, heteroclinic orbits \cite{conley:cbms},  chaotic dynamics \cite{mischaikow:mrozek:95,szymczak:96, day:junge:mischaikow, day:frongillo}, and semi-conjugacies onto nontrivial dynamics \cite{mischaikow:95, McCord:Mischaikow:96, mccord:00}.
Thus, the framework being considered in this paper is capable of identifying many of the building blocks of traditional nonlinear dynamics.

\begin{rem}
\label{rem:main_thm}
Theorem~\ref{thm:intro:main} guarantees that given the constraint that $\rho$ is sufficiently small there is a constructive method for identifying the desired dynamics.
This may appear to be a weak result since no quantitative bounds are provided for $\rho$.
However, the complexity of the geometry of attractors and the boundaries of their domains of attraction can be extreme \cite{kennedy:yorke}, and this geometry plays an intrinsic role in the necessary value of $\rho$.
Given the level of generality at which we are working (the only essential assumption is that $f$ is continuous) we see no path forward, based on finite data, to provide a sharp quantification of potential values of $\rho$.
\end{rem}

As presented in Theorem~\ref{thm:intro:main} our result can be viewed as an abstract approximation theorem.
To emphasize that {\bf P1} - {\bf P3} can (at least conceptually) be realized via some form of machine learning we recall two theorems associated with feedforward neural networks, also known as multilayer perceptrons.
The first theorem is a classical result from machine learning, the universal approximation theorem \cite{cybenko, hornik, leshno:lin:pinkus:schocken}, which guarantees that ${\bf P1}$ can be achieved.

\begin{theorem}
\label{thm:universalApproximation}
Consider a compact set $X\subset \R^d$ and assume that $f\colon X\to X$.
Given $\rho >0$ there exists a fully-connected feedforward neural network $G \colon \R^{d} \to \R^{d}$  such that $\sup_{x\in X}\| f(x) - G(x)\|<\rho$.    
\end{theorem}

To guarantee that {\bf P3} can be realized
we add the restriction that the activation functions are non-constant and piecewise linear (the results of \cite{leshno:lin:pinkus:schocken} guarantee that Theorem~\ref{thm:universalApproximation} still apply). This gives rise to approximating functions of the following form.

\begin{theorem}[Grigsby and Lindsey \cite{grigsby:lindsey}]
Let $G\colon \R^d \to \R^d$ be a fully-connected feedforward neural network with piecewise linear activation. 
Then $G$ is affine-linear on the cells of a realization of $\R^d$ as a polyhedral complex $\cP$.
\label{thm:polyhedral-decomposition}    
\end{theorem}

Observe that theoretically, local Lipschitz constants can be obtained by evaluating $G$ on the cells of $\cP$. The work of \cite{fazlyab2019} provides an algorithm to compute an accurate upper bound for a global Lipschitz constant for $G$.

\begin{rem}
\label{rem:notGoal} 
As suggested in Remark~\ref{rem:main_thm} in this paper we do not address the question of how to achieve $G$ using a machine learning method.
Rather, our focus is on demonstrating that given $G$, it is possible to rigorously identify dynamics of $f$ using the framework that we are proposing.
\end{rem}

\begin{ex}
\label{ex:Gapprox}    
We use the family of maps introduced in Example~\ref{ex:bifurcation} to provide intuition about the content of Theorem~\ref{thm:intro:main}.
We leave it to the reader to check that $f_1([-2,2])= [0,1]\subset \Int([-2,2])$.
Consider (using the language of Theorem~\ref{thm:intro:main}) 
a Morse representation $(\sMR,\leq_{\sMR_{f_1}})$ for $f_1$ whose elements, called Morse sets, are the fixed points (invariant sets) $x=0$ and $x=1$, and for which the partial order is $0 \leq_{\sMR_{f_1}} 1$.
The partial order captures the fact that if $x\in (0,1)$ then $\lim_{n\to \infty}f_1^n(x) = 0$ and $\lim_{n\to -\infty}f_1^n(x) = 1$.
The Morse sets identify the recurrent dynamics of $f_1$ and the partial order characterizes the non-recurrent (gradient-like) dynamics of $f_1$.
The Conley index of $0$ is that of an attracting fixed point, while the Conley index of $1$ is trivial.

Theorem~\ref{thm:intro:main} states that if $\rho>0$ is sufficiently small and $\sup_{x\in[-2,2]}\|f(x) - G(x)\| < \rho$, then we can use $G$ to recover essential information about the dynamics of $f_1$.

As discussed above, in general, invariant sets (Morse sets) are not computable, and therefore, Morse representations are not computable.
What are computable are Morse graphs (see Definition~\ref{defn:MorseGraph}).
Theorem~\ref{thm:intro:main}(i) provides the link between the computable Morse graph and the desired Morse representation.
The computation of the  Morse graph $\sMG$ is based on the assumption that $X$ is a polygonal region, {\bf P2} and {\bf P3}, and a discretization of the phase space. For this example the discretization is composed of intervals.
This allows us to produce a combinatorial representation of the dynamics in the form of a directed graph $\cG$ (see Sections~\ref{sec:MorseGraphs} and \ref{sec:combinatorial_rep_of_dyn}) where the vertices correspond to the above mentioned intervals and the directed edges indicate how intervals are mapped amongst each other.
The nodes of the associated Morse graph capture recurrence within the directed graph and the order relation provides information about the existence of nonrecurrent paths. 

To gain intuition as to why a combinatorial representation is sufficient, assume (this is not true in general) that the approximation $G$ equals $f_\theta$ for some $\theta\approx 1$.
If $G=f_1$, then there is no surprise that we can capture the dynamics of $f_1$.

So, assume that $G=f_\theta$ for $\theta < 1$. 
By assumption $\sup_{x\in[-2,2]}\|f(x) - G(x)\| < \rho$.
Let $I_0$ denote the interval that contains $0$.
Then $G(I_0)\cap I_0 \neq\emptyset$.
Thus, under $G$ the interval $I_0$ that contains $0$ clearly maps to itself, and hence the directed graph $\cG$ contains the self edge $I_0 \to I_0$.
Let $I_1$ denote the interval that contains $1$.
Similarly, (since $\rho$ is assumed small or equivalently $\theta \approx 1$) the above mentioned bounds  implies that $G(I_1)\cap I_1 \neq\emptyset$ and hence $\cG$ contains the self edge $I_1 \to I_1$. 
Ideally, intervals between $I_0$ and $I_1$ are mapped strictly to the left, in which case the Morse graph $\sMG$ has the form $I_1 \rightarrow I_0$.
The poset morphism $\nu\colon \sMG\to \sMR$ has the form $\nu(I_0) =0$ and $\nu(I_1) =1$, and in this idealized case $\nu$ is an isomorphism. 

As described in Section~\ref{sec:OuterApproximations}, $\cG$ can be used to compute Conley indices and furthermore these Conley indices are valid for the continuous dynamics generated by $G$.
The Conley index  associated with $I_0$ is that of an attracting fixed point.
The Conley index associated with $I_1$ is trivial; it must be since under $G$ the invariant set associated with $I_1$ is the empty set.
The fact that $\nu$ is a poset morphism allows us to conclude that the Conley indices for the fixed points $0$ and $1$ of $f_1$ are that of an attracting fixed point and trivial, respectively.
Referring again to \cite{srzednicki:85,mccord:88}, knowledge that the Conley index of $\nu(I_0)$ is that of a fixed point allows us to conclude that the invariant set associated with $\nu(I_0)$ contains a fixed point.
As emphasized above the fact that the Conley index of $\nu(I_1)$ is trivial does not imply that the invariant set associated with $\nu(I_1)$ is empty.
The  Morse graph $\sMG$ identified $I_1$ and the claim that $\nu$ is an epimorphism implies that $\nu(I_1)$ is a non-empty invariant set, but suggests that it may not be robust with respect to perturbations.

For a variety of reasons $\nu$ need not be an isomorphism (these reasons are what make the proof of Theorem~\ref{thm:intro:main} challenging).
Again, with the goal of gaining intuition, assume that $G=f_\theta$ for $\theta > 1$. 
The explicit form of $G$ (see Figure~\ref{fig:ftheta}(c)) and our assumption that $G$ satisfies 
{\bf P1} - {\bf P3}
allows us to deduce that there are distinct intervals $I_0$, containing $0$, $I_1$, containing $1$, and $I_\theta$, containing $\theta$, on which $G$ is affine-linear.
The combinatorial representation of the dynamics of $G$ must contain the self edges $I_0\to I_0$, $I_1\to I_1$, and $I_\theta\to I_\theta$.

With regard to additional edges one possibility is that the above mentioned bounds force the existence of $I_1 \to I_\theta$ and $I_\theta \to I_1$.
This implies that there is recurrence between $I_1$ and $I_\theta$, and thus they are identified as a single node in the Morse graph $\sMG$.
Again, assuming an ideal setting this leads to $\nu\colon \sMG \to \sMR$ being an isomorphism as above.

However, an alternative option is that  $I_1 \to I_\theta$, but there is no edge from $I_\theta$ to $I_1$.
In this case the simplest Morse graph consists of nodes $I_0$, $I_1$, and $I_\theta$ with the ordering $I_0 \leq_\sMG I_1$ and $I_\theta \leq_\sMG I_1$.
In this case $\nu(I_0) = 0$, $\nu(I_1) = 1$, and $\nu(I_\theta) = 1$ is the poset epimorphism guaranteed by Theorem~\ref{thm:intro:main}(i).
Observe that in this case we have learned dynamics that is more complicated than that expressed by $f$, however the existence of $\nu$ guarantees that we can recover $\sMR$, the chosen Morse representation of $f$.

As indicated above, $\cG$ can be used to compute Conley indices.
For this case the Conley indices of $I_0$ and $I_\theta$ are that of a stable fixed point, while the Conley index of $I_1$ is that of a fixed point with an unstable manifold of dimension one. 
However, $\cG$ also allows us to determine that under $f_1$ the Conley indices of $0$ and $1$ are those of a stable fixed point and trivial, respectively.
\end{ex}

\begin{rem}
\label{rem:applications}
We remind the reader that Theorem~\ref{thm:intro:main} demonstrates that using the coarser framework of dynamics of this paper, i.e., a Morse representation, dynamics generated by a continuous function can be recovered by a fully-connected feedforward neural network $G$.
In the context of applications, we do not know $f$ and hence cannot know a Morse representation $(\sMR,\leq_{\sMR_f})$.
In practice, $G$ is learned from data, $\cG$ is constructed, and then  conclusions are drawn about the dynamics of $f$.
We return to this point in Section~\ref{sec:examples}.
\end{rem}

An outline of the paper is as follows. In Section~\ref{sec:ConleyThy}, we introduce Conley theory for continuous maps and describe how dynamics can be characterized in this framework. To demonstrate the generalizability of this framework, we present Theorem \ref{thm:intro:generalizable}, which loosely states that
for any continuous function $g$ sufficiently close to a continuous function $f$, the characterization of dynamics of $f$ also characterizes the dynamics of $g$. In Section~\ref{sec:MorseGraphs}, we introduce combinatorial dynamics and describe how a Morse graph is computed.
In Section~\ref{sec:combinatorial_rep_of_dyn}, we connect this perspective back to continuous dynamics and discuss combinatorial representations of dynamics obtained from approximations.
In Section~\ref{sec:proof}, we prove Theorem~\ref{thm:intro:main}. 
We conclude in Section~\ref{sec:examples} with a demonstration of the applicability of this approach.

\section{Nonlinear Dynamics}
\label{sec:ConleyThy}

This section provides an introduction to Conley theory in the context of  dynamics generated by continuous functions.

\subsection{Continuous Dynamics}
\label{sec:ContinuousDynamics}

We consider continuous functions $f\colon  \R^d\to \R^d$  and restrict our attention to the dynamics on a
finite union of closed, bounded, convex $d$-dimensional polytopes $X\subset \R^d$ for which $f(X) \subset \Int(X)$.
However, the discussion of this section applies to continuous maps on compact metric spaces \cite{conley:cbms, mischaikow:mrozek}.

We do not assume that $f$ is invertible, and therefore, some care must be taken to discuss trajectories of $f$. 
In particular, given $x\in X$, $\gamma_x\colon \Z \to X$ is a \emph{full trajectory through} $x$ if $\gamma_x(0)=x$ and  $\gamma_x(n+1) = f(\gamma_x(n))$ for all $n \in \Z$.
As indicated in the introduction, $S\subset X$ is an invariant set of $f$ if $f(S)=S$.  This is equivalent to requiring that for every $x\in S$, there exists a full trajectory $\gamma_x \colon \Z \to S$.

\begin{defn}
\label{defn:MaxInvSet}
Let $N\subset X$. 
The \emph{maximal invariant set in $N$} under $f$ is
 \[
 \Inv(N,f):= \setof{x \in N \mid \exists \text{ a full trajectory }\gamma_x: \mathbb{Z} \to N}.
 \]
\end{defn}

\begin{defn}
An invariant set $A\subset X$ is an \emph{attractor} if there exists a compact set $N\subset X$ such that 
\begin{equation}
\label{eq:att}
A = \omega(N,f) := \bigcap_{n\geq 0}\cl\left(\bigcup_{k=n}^\infty f^k(N)\right) \subset \Int(N)
\end{equation}
where $\cl$ and $\Int$ respectively denote closure and interior.
\end{defn}
The set of all attractors of $f$ is denoted by $\sAtt(f)$ and forms a bounded distributive lattice under the operations $\vee = \cup$ and $A_1 \wedge A_2 = \omega(A_1 \cap A_2)$
\cite{kalies:mischaikow:vandervorst:14} (for basic definitions associated with lattices see Section~SM1 in the Supplementary Materials).
The minimal element is $\bzero = \emptyset$ and the maximal element is $\bone = \omega(X,f)$.
Given $A\in\sAtt(f)$, the \textit{dual repeller} to $A$ is defined to be $A^*:=\{x\in \bone\mid \omega(x,f)\cap A = \emptyset\}$.

\begin{ex}
\label{ex:att(f)}
Recall $f_\theta$ defined in Example~\ref{ex:bifurcation}.
We leave it to the reader to check that
\[
\bone_\theta := \omega([-2,2],f_\theta) = \begin{cases}
    \setof{0} & \text{if $\theta < 1$}  \\
    [0,\theta] & \text{if $1\leq \theta$.}
\end{cases}
\]
The dynamics of $f_\theta$ is simple, so $\sAtt(f_\theta)$ is finite.
In particular,
\[
\sAtt(f_\theta) = \begin{cases}
    \setof{\emptyset, \bone_\theta}  & \text{if $\theta < 1$,}  \\
    \setof{\emptyset, \setof{0}, \bone_1} & \text{if $\theta = 1$,} \\
    \setof{\emptyset, \setof{0}, \setof{\theta}, \setof{0, \theta}, \bone_\theta} & \text{if $1< \theta$.}
\end{cases}
\]
Attractors are invariant sets, and therefore, it is not  surprising that  $\sAtt(f_\theta)$  changes abruptly as a function of $\theta$.
\end{ex}

Conley's decomposition theorem \cite{conley:cbms,robinson} states that
$CR(f) = \bigcap_{A\in\sAtt(f)} A\cup A^*$
is an invariant set, called the \emph{chain recurrent set}.
Furthermore, all recurrent dynamics of $f$ occurs in $CR(f)$ and $f$ exhibits gradient-like dynamics in the complement of $CR(f)$.

In general $\sAtt(f)$ may be countably infinite \cite{conley:cbms}.
With this in mind, we restrict our attention to  finite sublattices of $\sAtt(f)$ containing both $\bzero$ and $\bone$.
We leave it to the reader to check that if $\sA$ is such a finite sublattice, then $CR(f)\subset \bigcap_{A\in\sA} A\cup A^*$.

\begin{defn}
\label{defn:MorseRepresentation}
A finite sublattice $\sA$ gives rise to a \emph{Morse representation}
\begin{equation}
\label{eq:AtoMR}
(\sMR(\sA),\leq_{\sMR_f}) :=\{A \cap (A^<)^* \mid A \in \sJ^\vee(\sA) \} 
\end{equation}
where $\sJ^\vee(\sA)$ denotes the join-irreducible elements of $\sA$ and $A^<$ denotes the immediate predecessor of $A \in \sJ^\vee(\sA)$
\cite{kalies:mischaikow:vandervorst:21}.
The elements of $\sMR(\sA)$, called \emph{Morse sets}, are nonempty, pairwise disjoint invariant sets.     
\end{defn}

Observe that $(\sMR(\sA),\leq_{\sMR_f})$ is finite since $\sA$ is finite.
The partial order $\leq_{\sMR_f}$ is defined via the minimal transitive extension of the following relation.
Given distinct Morse sets $M,M'\in \sMR(\sA)$, set $M <_{\sMR_f} M'$ if there exists $x\in X$ such that $\omega(x,f)\subset M$ and there exists a full trajectory $\gamma_x\colon \Z \to \bone$ such that $\bigcap_{n>0} \cl(\gamma_x(-\infty,-n]) \subset M'$.
Observe that the partial order $\leq_{\sMR_f}$ provides explicit information about the gradient-like structure of the dynamics; the dynamics moves from higher-ordered Morse sets to lower-ordered Morse sets.

\begin{rem}
    \label{rem:MRtoAtt}
As presented above, a finite sublattice of $\sAtt(f)$ gives rise to a well defined Morse representation.
The converse is also true, i.e., by \cite[Theorem 6]{kalies:mischaikow:vandervorst:21} given a Morse representation $\sMR$ of $f$ there is a unique finite lattice of attractors $\sA$ such that \eqref{eq:AtoMR} holds.
In particular, \eqref{eq:AtoMR} gives rise to a poset isomorphism defined by
\begin{align*}
   \psi\colon \sMR &\to \sJ^\vee(\sA) \\
   M & \mapsto A
\end{align*}
for $M = A\cap (A^<)^*$.
\end{rem}

\begin{ex}
\label{ex:A}   
A given dynamical system can have multiple Morse representations.
The lattice of attractors $\sAtt(f_\theta)$ gives rise to the following Morse representation: 
\[
\sMR(\sAtt(f_\theta),\leq_{f_\theta}) = \begin{cases}
    0 & \text{if $\theta < 1$} \\
    (\setof{0, 1}, 0<_{f_\theta} 1) & \text{if $\theta =1$} \\
    (\setof{0,1,\theta},0<_{f_\theta} 1 >_{f_\theta} \theta) & \text{if $\theta > 1$.} 
\end{cases}
\]
An alternative lattice of attractors for $\theta \geq 1$ is 
$\sA_\theta := \setof{0,[0,\theta]}$
in which case
$\sMR(\sA_\theta,\leq_{f_\theta}) = (\setof{0,[1,\theta]},0<_{f_\theta} [1,\theta])$.
Both Morse representations provide a decomposition of the dynamics of $f_\theta$. 
The existence of the order preserving epimorphism $\mu\colon \sMR(\sAtt(f_\theta))\to \sMR(\sA_\theta)$ defined by $\mu(\setof{0}) = \setof{0}$ and $\mu(\setof{1}) = \mu(\setof{\theta}) = [1,\theta]$ indicates that the decomposition provided by $\sMR(\sAtt(f_\theta))$ is finer  than that of $\sMR(\sA_\theta)$.
In this sense the choice of Morse representation can be viewed as a choice of the scale at which one wants to understand the dynamics.
\end{ex}

As Example~\ref{ex:A} demonstrates, Morse representations need not be stable with respect to perturbations, and thus, cannot be expected to be computable in general.

The partial order associated with a Morse representation provides a characterization of the structure of the non-recurrent dynamics. 
To identify finer structures we make use of the Conley index that we introduce via an extremely terse series of definitions. 
For more detailed discussion we refer the reader to \cite{conley:cbms, robbin:salamon, mrozek:99, szymczak:95, franks:richeson, mischaikow:mrozek, mischaikow:weibel}.

\begin{defn}
\label{defn:isolatingNeighborhood}
A subset $N\subset X$ is an \emph{isolating neighborhood} if $\Inv(\cl(N),f)\subset\Int(N)$.
An invariant set $S$ is \emph{isolated} if there exists an isolating neighborhood $N$ such that $S= \Inv(\cl(N),f)$.
\end{defn}

\begin{defn}
\label{defn:indexPair}
Let $S$ be an isolated invariant set of $f$.
Consider a pair $P = (P_1,P_0)$ of compact sets satisfying $P_0\subset P_1 \subset X$ with the property that $P_1\setminus P_0$ is an isolating neighborhood for $S$.
Let $(P_1/P_0, [P_0])$ denoted the pointed quotient space in which $P_0$ is collapsed to a point.
Define $f_P\colon (P_1/P_0, [P_0]) \to (P_1/P_0, [P_0])$ by
\[
f_P(x) := \begin{cases}
    f(x) & \text{if $x, f(x) \in P_1\setminus P_0$}\\
    [P_0] & \text{otherwise.}
\end{cases}
\]
If $f_P$ is continuous, then $P$ is called an \emph{index pair}  and $f_P$ is called an \emph{index map} for $S$.
\end{defn}

By definition if $P = (P_1,P_0)$ is an index pair, then $f_P$ is continuous, and thus, $f_P$ induces a map on homology denoted by $f_{P*}$.

\begin{defn}
    \label{defn:ConleyIndex}
The \emph{(homology) Conley index} $\sCon_*(S)$ of an isolated invariant set $S$ is the shift equivalence class of 
\[
f_{P*} \colon H_*(P_1/P_0, [P_0]) \to H_*(P_1/P_0, [P_0])
\]
where $P = (P_1,P_0)$ is an index pair for $S$.
\end{defn}

For the definition of shift equivalence, the reader is referred to \cite{lind:marcus, boyle, franks:richeson, mischaikow:weibel}.
Over fields, the homology Conley index can be characterized more simply by the rational canonical form of $f_{P*}$, excluding nilpotent blocks \cite{mischaikow:weibel}, \cite[7.3–7.5]{lind:marcus}. 
Examples of computing the homology Conley index are contained in \cite{bush:cowan:harker:mischaikow} and use an efficient algorithm due to Storjohann \cite{storjohann}.
For the purposes of this paper, it is sufficient to know that in the context of the homology Conley index, shift equivalence is computable \cite{mischaikow:weibel}.

The Conley index of $S$ is \emph{trivial}, denoted by $\sCon_*(S) = 0$, if $f_{P*}$ is nilpotent, in which case it is shift equivalent to the zero map.
Rephrasing what is stated in the introduction, if $\sCon_*(S) \neq 0$, then $S\neq \emptyset$.

\subsection{Computable Objects}
\label{sec:ComputableObjects}

The previous subsection focused on the study of invariant sets via Conley theory.
In this subsection we focus on computable objects.

\begin{defn}
\label{defn:attractingblock}
A compact set $N\subset X$ is an \emph{attracting block} for $f$ if $f(N)\subset \Int(N)$.
\end{defn}

Let $\sABlock(X, f)$ denote the set of all attracting blocks for $f$.
The following proposition follows from the work of \cite{kalies:mischaikow:vandervorst:14}. We provide an elementary proof to make this paper more self-contained.

\begin{proposition}\label{prop:ablock_is_bdlat}
Let $X$ be a regular compact subset of $\R^d$.
Let $f$ be a continuous function such that $f(X)\subset \Int(X)$.
Then, the set $\sABlock(X, f)$ is a bounded, distributive lattice with respect to $\vee = \cup$ and $\wedge = \cap$. 
The neutral elements of $\sABlock(X, f)$ are $\bone_b = X$ and $\bzero_b = \varnothing$.
\end{proposition}

\begin{proof}
 Let $U_1, U_2 \in \sABlock(X, f)$. By definition, $U_1$ and $U_2$ are compact. Since $X$ is a metric space, $U_1 \cap U_2$ and $U_2 \cup U_2$ are compact. Furthermore,
 \begin{align*}
     f(U_1 \cup U_2) & = f(U_1) \cup f(U_2)\\
     & \subset \Int(U_1) \cup \Int(U_2)\\
     & \subset \Int(U_1 \cup U_2),
 \end{align*}
 which proves that $U_1 \cup U_2 \in \sABlock(X, f)$.

 Also,
 \begin{align*}
     f(U_1 \cap U_2) & \subset f(U_1) \cap f(U_2)\\ 
     & \subset \Int(U_1) \cap \Int(U_2)\\
     & = \Int(U_1 \cap U_2),
 \end{align*}
 which proves that $U_1 \cap U_2 \in \sABlock(X, f)$.

 The elements $X$ and $\varnothing$ are attracting blocks that satisfy the definitions of the maximal and minimal elements, respectively. Distributivity follows from the fact that $\sABlock(X, f)$ is a sublattice of the powerset lattice for $X$, which is distributive.
\end{proof}

\begin{rem}
Throughout this paper, a finite sublattice $\sN$ of $\sABlock(f)$ is assumed to contain the minimal and maximal elements $\bzero_b := \emptyset$ and $\bone_b := X$. 
\end{rem}

\begin{rem}
The fact that the lattice operations on $\sABlock(f)$ are $\cap$ and $\cup$ implies that the partial order relation on $\sABlock(f)$ is given by inclusion.
\end{rem}

\begin{rem}
\label{rem:kpsi}
Let $\sA$ be a finite sublattice of $\sAtt(f)$.
By \cite[Theorem 1.2]{kalies:mischaikow:vandervorst:15} there exists a sublattice $\sN$ of $\sABlock(f)$ and a lattice isomorphism $k\colon \sA \to \sN$ such that $\omega(k(A),f)=A$ for all $A \in \sA$.
Furthermore, consider a Morse representation $\sMR(\sA)$ and the corresponding isomorphism $\psi\colon \sMR(\sA)\to \sJ^\vee(\sA)$ of Remark~\ref{rem:MRtoAtt}.
Then, for every $M\in \sMR(\sA)$, $M= \Inv(k(\psi(M))\setminus k(\psi(M)^<),f)$.
\end{rem}

\begin{ex}
\label{ex:Ablock}   
Consider $f_\theta\colon [-2,2]\to [-2,2]$ as defined in Example~\ref{ex:bifurcation} and the Morse representation $\sMR(\sAtt(f_\theta),\leq_{f_\theta})$ of Example~\ref{ex:A}.
We leave it to the reader to check that for $\theta = 1.5$,
\[
\sN = \setof{\emptyset,[-2,0.25],[1.25,2], [-2,0.25]\cup [1.25,2], [-2,2]}
\]
is a sublattice of $\sABlock(f_{1.5})$ and $k\colon \sAtt(f_{1.25})\to \sN$ generated by $k(0) = [-2,0.25]$ and $k(\theta) = [1.25,2]$ is a lattice isomorphism.
Furthermore, $0 = \Inv([-2,0.25],f_{1.5})$,  $\theta = \Inv([1.25,2],f_{1.5})$, and $1 = \Inv([-2,2] \setminus ([-2,0.25]\cup [1.25,2]) ,f_{1.5})$.

We do not claim that $\sN$ is the unique lattice of attracting blocks with this property.
As an example consider
\[
\sN' = \setof{\emptyset,[-2,0.15],[1.35,2], [-2,0.15]\cup [1.35,2], [-2,2]}
\]
with $k'$ generated by $k'(0) = [-2,0.15]$ and $k'(\theta) = [1.35,2]$. 
\end{ex}

\begin{rem}
\label{rem:ablock_stable}
A result of fundamental importance is that attracting blocks are stable with respect to continuous perturbations.
To be more precise, let $N$ be an attracting block for $f$.
There exists $\epsilon >0$ such that if $g\colon X\to X$ is a continuous function and $\sup_{x\in X} \| f(x) -g(x)\| < \epsilon$, then $N$ is an attracting block for $g$.
Since the lattice $\sN$ consists of a finite number of attracting blocks, the same argument guarantees that $\sN$ is stable with respect to perturbations.
\end{rem}

This stability is exhibited in Example~\ref{ex:Ablock} where $\sN$ is a lattice of $\sABlock(f_\theta)$ for $1.25 < \theta < 1.75$.

The following result follows from \cite[Corollary 4.4]{robbin:salamon}.
\begin{proposition}
\label{prop:IndexPair}
Consider a finite sublattice $\sN\subset\sABlock(f)$.
Let $N_1,N_0\in\sN$ with $N_0\subset N_1$.
Then $(N_1,N_0)$ is an index pair.
\end{proposition}

\begin{rem}
\label{rem:indexpair}
In the context of this paper index pairs of particular importance arise as follows.
First, consider the setting of Remark~\ref{rem:kpsi} for which we have a Morse representation $\sMR(\sA)$, an isomorphism  $k\colon \sA \to \sN$, and an isomorphism $\psi\colon \sMR(\sA) \to \sJ^\vee(\sA)$.
Then, for each $M\in\sMR(\sA)$, $(k(\psi(M)),k(\psi(M)^<))$ is an index pair and $\sCon_*(M,f)$, the Conley index of $M$ under $f$, is given by the shift equivalence of $f_*\colon H_*(k(\psi(M)),k(\psi(M)^<)) \to H_*(k(\psi(M)),k(\psi(M)^<))$.
\end{rem}
 
As indicated in Proposition~\ref{prop:IndexPair}, nested pairs of attracting blocks form index pairs and give rise to isolating neighborhoods.
We impose structure on this observation as follows.

\begin{defn}
\label{defn:MorseTiles}
Given a finite sublattice $\sN\subset \sABlock(f)$, the associated \emph{Morse tiles} are
\[
\sMTile(\sN) := \setdef{T(N):= \cl(N\setminus N^<)}{N\in \sJ^\vee(\sN)}.
\]
Morse tiles form a poset $(\sMTile(\sN), \leq_\sMT)$ where the partial order $\leq_\sMT$ is inherited from the partial ordering of $\sJ^\vee(\sN)$.
\end{defn}

\begin{rem}
\label{rem:isoMTandJN}
By construction, $\sMTile$ and $\sJ^\vee(\sN)$ are order-isomorphic.    
\end{rem}

\begin{defn}
\label{defn:MorseDecomposition}
A \emph{Morse decomposition} of $f$ is an order-embedding \\
$\mu\colon (\sMR,\leq_{\sMR_f}) \to (\sP,\leq_P)$,
where $(\sMR,\leq_{\sMR_f})$ is a Morse representation of $f$ and $(\sP,\leq_P)$ is a finite poset.
\end{defn}

While the definition of a Morse decomposition may appear somewhat mysterious, as indicated by the following proposition it arises naturally.

\begin{proposition}
\label{prop:MDexist}
Consider a finite sublattice $\sN\subset\sABlock(f)$.
Define\\ $(\sMR,\leq_{\sMR_f})$ by
\[
\sMR = \setdef{M(N):=\Inv(N\setminus N^<,f)\neq\emptyset}{N\in \sJ^\vee(\sN)},
\]
where $\leq_{\sMR_f}$ is inherited from the order on $\sJ^\vee(\sN)$.
Then, $\sMR$ is a Morse representation of $f$ and the embedding $\mu \colon \sMR \to \sJ^\vee(\sN)$ given by $\mu(M(N))=N$ is a Morse decomposition.
\end{proposition}

\begin{ex}
\label{ex:ConleyIndex}
Consider the lattice of attracting blocks $\sN$ of Example~\ref{ex:Ablock}.
The associated Morse tiling is
\[
\sMTile(\sN) = \setof{[-2,0.25], [1.25,2], \cl([-2,2]\setminus([-2,0.25] \cup [1.25,2]))}.
\]
We use $\sMTile(\sN)$ to compute the Conley indices of
$\Inv([-2,0.25],f_{1.5})$,\\
$\Inv([1.25,2],f_{1.5})$, and $\Inv([-2,2]\setminus([-2,0.25] \cup [1.25,2]),f_{1.5})$.
In particular, using field  coefficients $\F$
\begin{align*}
    H_k([-2,0.25],\emptyset;\F) \cong H_k([1.25,2],\emptyset;\F) & \cong \begin{cases}
        \F & \text{if $k=0$}\\
        0 & \text{otherwise.}
    \end{cases} \\
    H_k([-2,2],[-2,0.25] \cup [1.25,2];\F) &\cong \begin{cases}
        \F & \text{if $k=1$}\\
        0 & \text{otherwise.}
    \end{cases}
\end{align*}
Since we are using field coefficients, determining shift equivalence (which we denote by $\sim$) reduces to determining rational canonical form, thus
\begin{align*}
\sCon_k(\Inv([-2,0.25],f_{1.5});\F) & \sim \sCon_k(\Inv([1.25,2],f_{1.5});\F)  \\
& \sim \begin{cases}
    [1] \colon \F \to \F & \text{if $k=0$} \\
    0 & \text{otherwise}
\end{cases}\\
\sCon_k(\Inv([-2,2]\setminus([-2,0.25] \cup [1.25,2]),f_{1.5});\F)& \sim  \begin{cases}
    [1] \colon \F \to \F & \text{if $k=1$} \\
    0 & \text{otherwise.}
\end{cases}\\
\end{align*}

For this example it is clear that
$0 =\Inv([-2,0.25],f_{1.5})$, $\theta =\Inv([1.25,2],f_{1.5})$, and $1 =\Inv([-2,2]\setminus([-2,0.25] \cup [1.25,2]),f_{1.5})$.
However, for general nonlinearities even identifying the existence of an invariant set is typically a nontrivial challenge.
Thus, it is worth noting  that the nontriviality of the Conley indices implies that the associated invariant sets are nonempty. 
\end{ex}

\begin{ex}
\label{ex:ConleyIndex2}
We revisit Example~\ref{ex:ConleyIndex}, but with a coarser lattice of attracting blocks.
Let $\sN'' := \setof{\emptyset, [-2,0.25], [-2,2]}$.
Then
\[
\sMTile(\sN'') = \setof{[-2,0.25], \cl([-2,2]\setminus [-2,0.25])}
\]
and
\begin{align*}
    H_k([-2,0.25],\emptyset;\F) & \cong \begin{cases}
        \F & \text{if $k=0$}\\
        0 & \text{otherwise.}
    \end{cases} \\
    H_*([-2,2],[-2,0.25] ;\F) & = 0.
\end{align*}
Thus,
\begin{align*}
\sCon_k(\Inv([-2,0.25],f_{1.5});\F) 
& \sim \begin{cases}
    [1] \colon \F \to \F & \text{if $k=0$} \\
    0 & \text{otherwise}
\end{cases}\\
\sCon_*(\Inv([-2,2]\setminus [-2,0.25],f_{1.5});\F) &= 0.
\end{align*}
Observe that unlike in Example~\ref{ex:ConleyIndex}, we cannot use the Conley index to identify that $[1,1.5] = \Inv([-2,2]\setminus [-2,0.25],f_{1.5})$ is nonempty.
\end{ex}

\subsection{Summary of Theory}

The goal of this subsection is to indicate how concepts from Sections~\ref{sec:ContinuousDynamics} and \ref{sec:ComputableObjects} suggest the possibility of Theorem~\ref{thm:intro:main}.
We begin with a theorem, an imprecise statement of which is as follows.
Assume that we have obtained a characterization of the dynamics of $f$ using the framework of Conley theory. 
Then, for any continuous function $g$ sufficiently close to $f$, the characterization of dynamics of $f$ characterizes the dynamics of $g$.

For the sake of completeness we provide a formal proof of Theorem~\ref{thm:intro:generalizable}, even though it is essentially a reformulation of classical results.

\begin{theorem}
\label{thm:intro:generalizable}
Let $f\colon X\to X$ be a continuous function on a compact subset of $\R^d$ and assume that $f(X)\subset \Int(X)$.
Let $(\sMR_f,\leq_{\sMR_f})$ be a Morse representation of $f$.
Furthermore, for each Morse set $M_f\in \sMR$, let $\sCon_*(M_f,f)$ denote the homology Conley index of $M_f$ under the dynamics induced by $f$.

Then, there exists $\epsilon >0$ for which the following result holds.
If $g\colon X\to X$ is a continuous function such that $\sup_{x\in X}\| f(x) - g(x)\|<\epsilon$, then there is a Morse decomposition for $g$, i.e., a monomorphism of partially ordered sets
\[
\mu\colon (\sMR_g,\leq_{\sMR_g}) \to (\sMR_f,\leq_{\sMR_f})
\]
where $\sMR_g$ is a Morse representation of $g$, such that if $M_g\in \sMR_g$, then $\sCon_*(M_g, g)$ is equivalent to $\sCon_*(\mu(M_g), f)$.
\end{theorem}

\begin{proof}
Let $(\sMR_f,\leq_{\sMR_f})$ be a Morse representation of $f$ with minimal partial order $\leq_{\sMR_f}$. By \cite[Theorem 6]{kalies:mischaikow:vandervorst:21} there exists a unique finite sublattice of $\sAtt(f)$ and a poset isomorphism $\psi\colon (\sMR_f,\leq_{\sMR_f})\to (\sJ^\vee(\sA),\subset)$. 
By \cite[Theorem 1.2]{kalies:mischaikow:vandervorst:15} there exists a sublattice $\sN$ of $\sABlock(f)$ and a lattice isomorphism $k\colon \sA \to \sN$ such that $\omega(k(A),f)=A$.
This implies there is a poset isomorphism $\kappa: \sJ^\vee(\sA) \to \sJ^\vee(\sN)$.

Recall that by definition there exists a poset isomorphism $\iota\colon \sJ^\vee(\sN)\to \sMTile(\sN)$.
Observe that by construction if $M\in \sMR_f$, then $M=\Inv(\iota(\kappa(\psi(M)),f)$.

As indicated by Remark~\ref{rem:ablock_stable}, for each $N\in \sN$ there exists $\epsilon_N >0$ such that if $\sup_{x\in X}\| f(x) - g(x)\|<\epsilon_N$, then $N$ is an attracting block for $g$.
Set $\epsilon = \min\setdef{\epsilon_N}{N\in\sN}$.
Then $\sN$ is a finite sublattice of $\sABlock(g)$.
By Proposition~\ref{prop:MDexist} we have the existence of a Morse decomposition $\mu\colon\sMR_g \to \sJ^\vee(\sN)$ where
\[
\sMR_g = \setdef{M(N):=\Inv(N\setminus N^<,g)\neq\emptyset}{N\in \sJ^\vee(\sN)}.
\]
Since $\sJ^\vee(\sN) \cong \sJ^\vee(\sA) \cong \sMR_f$, we have the desired Morse decomposition.

Let $M_g\in \sMR_g$. Then by definition, $M_g=\Inv(N\setminus N^<,g)$ for some $N\in \sJ^\vee(\sN)$.
Since $\cl(N\setminus N^<)$ is an isolating neighborhood for all functions $g\colon X\to X$ such that $\sup_{x\in X}\| f(x) - g(x)\|<\epsilon$, by \cite{franks:richeson}
\[
\Con(M_g,g)\cong \Con(\mu(M_g),f).
\]
\end{proof}

Theorem~\ref{thm:intro:generalizable} implies that we have a robust framework for characterizing dynamics. 
Given knowledge of a Morse representation of $f$, we can deduce information about the Morse representation of $g$ for $g$ sufficiently close to $f$.
However, it does not state that we can learn the dynamics of $f$.
Thus, we turn to the main result of this paper.

\section{Combinatorial Dynamics}
\label{sec:MorseGraphs}

Let $\cX$ be a finite set.
A \emph{combinatorial multivalued map} $\cF\colon \cX \mvmap \cX$ satisfies two conditions for every $\xi\in\cX$: (i) $\cF(\xi)\subset \cX$, and (ii) $\cF(\xi)\neq\emptyset$.

Observe that a combinatorial map $\cF$ is equivalent to a directed graph;  $\cX$ is the set of vertices and there exists a directed edge $\xi \to \xi'$ if and only if $\xi' \in \cF(\xi)$, with the additional constraint that there is a directed edge from $\xi$ for every $\xi\in\cX$. 
We use both perspectives in this section.

\begin{defn}
\label{defn:invset+} 
Let $\cF\colon \cX \mvmap \cX$ be a combinatorial multivalued map.
A \emph{forward invariant set} of $\cF$ is a set $\cN\subset \cX$ such that $\cF(\cN)\subset\cN$.
The set of forward invariant sets of $\cF$ is denoted by $\sInvset^+(\cF)$.
\end{defn}

As shown in \cite{kalies:mischaikow:vandervorst:15}, $\sInvset^+(\cF)$ is a distributive lattice under the operations $\vee = \cup$ and $\wedge = \cap$, with minimal element $\bzero = \emptyset$ and maximal element $\bone= \cX$.

We now describe how the Morse graph that appears in Theorem \ref{thm:intro:main} is computed.
A \textit{walk} in $\cF$ of length $K \in \N$ from $\xi_0 \in \cX$ to $\xi_K \in \cX$ is a sequence of vertices
 \[
 \{\xi_k \mid 0 \leq k \leq K \text{ and } \xi_k \in \cF(\xi_{k-1}) \text{ for all } 1 \leq k \leq K\}.
 \]
We denote a walk from $\xi_0$ to $\xi_K$ by $\xi_0 \walk \xi_K$. Under this definition, a loop is a walk of length one. For all vertices $\xi$, there exists a walk of length zero $\xi \walk \xi$.
 
Define the relation $\sim$ on $\cX$ by $\xi \sim \xi'$ if there exists walks $\xi \walk \xi'$ and $\xi' \walk \xi$. We note that $\sim$ is an equivalence relation. A directed graph $\cF: \cX \mvmap \cX$ is \textit{strongly connected} if the quotient set $\cX/\sim$ consists of one equivalence class, meaning there exists a walk between every pair of vertices.
 
Given an arbitrary directed graph $\cF: \cX \mvmap \cX$, the \textit{strongly connected components} are the induced subgraphs with vertex sets equal to equivalence classes of $\cX/\sim$. Equivalently, the strongly connected components are the maximal strongly connected subgraphs of $\cF$. 
Various linear time algorithms for computing strongly connected components have been developed: for example, Tarjan's algorithm \cite{tarjan} and Kosaraju-Sharir's algorithm \cite{aho:hopcroft:ullman, sharir}.

Let $\pi\colon \cX \to \sSC(\cF):= \cX/\sim$ denote the quotient map from $\cX$ to the equivalence classes of strongly connected components.
The graph obtained by contracting edges that have incident vertices in the same strongly connected component is the \textit{condensation graph} of $\cF: \cX \mvmap \cX$, denoted by $\bar{\cF}: \sSC(\cF) \mvmap \sSC(\cF)$. 
The condensation graph is an acyclic directed graph \cite{corman:leiserson:rivest:stein}, and it may be viewed as a compression of $\cF$ that retains reachability information. 
Due to acyclicity, $\sSC(\cF)$ is a poset under the relation $\zeta' \leq_{\sSC(\cF)} \zeta$ if $\zeta \walk \zeta'$ in $\bar{\cF}$.
A \textit{recurrent component} of $\cF: \cX \mvmap \cX$ is a strongly connected component with at least one edge in $\cF$. 

\begin{defn}
\label{defn:MorseGraph}
The \emph{Morse graph} of $\cF$, denoted by $(\sMG(\cF),\leq_\cF)$, is the subposet of $\sSC(\cF)$ consisting of recurrent components of $\sSC(\cF)$.
\end{defn}

\section{Combinatorial approximation of dynamics}
\label{sec:combinatorial_rep_of_dyn}

Throughout this section we assume that $X\subset \R^d$ is a
finite union of closed, bounded, convex $d$-dimensional polytopes and $f\colon \R^d\to \R^d$ is a continuous function with the property that $f(X) \subset \Int(X)$. 
The goal of this section is to introduce combinatorial approximations of $f$ from which Morse decompositions and Conley indices of $f$ can be derived.

\subsection{Simplicial Complexes}
\label{sec:simplicialComplex}

Throughout this section we assume that $\cX$ is a finite geometric simplicial complex in $\R^d$ and given $\xi\in\cX$ we set $|\xi| := \cl (\xi)\subset \R^d$.
Furthermore, we assume that $\cX$ satisfies the following two conditions:
\begin{description}
    \item[X1] $X\subset |\cX| := \bigcup_{\xi\in\cX}|\xi|$, and
    \item[X2] $X$ is a strong deformation retract of $|\cX|$.
\end{description}

The latter condition implies that $H_*(\cX) \cong H_*(X)$.
Since $X$ is a finite union of closed, bounded, convex $d$-dimensional polytopes and $\cX$ is a finite geometric simplicial complex, the top dimensional cells of $\cX$, denoted by $\cX^\text{top}$, are $d$-dimensional simplices.

The \emph{diameter} of $\cX$ is given by
\[
\diam(\cX):= \max_{\xi \in \cX}\setof{\diam(|\xi|)}.
\]

A set $Y\subset \R^d$ is \emph{acyclic} if 
\[
H_k(Y) \cong \begin{cases}
    \Z & \text{if $k=0$} \\
    0 & \text{otherwise.}
\end{cases}
\]
A collection $\mathcal{U}$ of $d$-dimensional simplicies (with boundary) is a \emph{good closed cover} of $X$ if $\cU^0:=\setdef{\Int(U)}{U \in \cU}$ is an open cover of $X$ and the intersection of any finite subcollection of $\cU$ is either empty or acyclic.
Let $\text{sd}^n(\cX)$ denote the simplicial complex obtained by performing $n$ barycentric subdivisions to $\cX$.
\begin{proposition}
\label{prop:triangulation}
Let $\cX$ denote a finite simplicial covering of $X$.
There exists $k \in \N$ such that the collection $\cU^o$ of open stars of vertices of $\text{sd}^k(\cX)$ is a finite open cover of $X$. Furthermore, $\cU:=\setdef{\cl(U)}{U \in \cU^o}$ is a finite good closed cover of $X$.
\end{proposition}
\begin{proof}
That $\cU$ is a good open cover of $X$ is straightforward.
Whether the closures form a good closed cover depends on the choice of triangulation $\cT$. 
A straightforward, but technical argument, shows that if one starts with $\cT^{(2)}$, the second barycentric subdivision of $\cT$, then $\cU$ has the desired property \cite{ramras}.
\end{proof}

\subsection{Outer Approximations}
\label{sec:OuterApproximations}

We use a combinatorial multivalued map\\
$\cF\colon \cX\mvmap \cX$ defined on a geometric simplicial complex $\cX$ to encode dynamics and set
\[
|\cF(\xi)| := \bigcup_{\xi' \in \cF(\xi)} |\xi'| \subset \R^d.
\]

To pass from combinatorial dynamics of multivalued maps to continuous dynamics we make use of the following construct.

\begin{defn}
\label{defn:outerApproximation}
Let $f\colon X\to X$ be a continuous function.
A combinatorial multivalued map $\cF\colon \cX \mvmap \cX$ is an \emph{outer approximation} of $f$ if for every $\xi\in \cX$
\[
f(| \xi |) \subset \Int(|\cF(\xi)|).
\]
If $\cF$ is an outer approximation of $f$, then $f$ is called a \emph{selector} of $\cF$.
\end{defn}

To compare combinatorial multivalued maps, we use the notion of enclosure.

\begin{defn}
A combinatorial multivalued map $\cF: \cX \mvmap \cX$ \emph{encloses} $\cF': \cX \mvmap \cX$ if for all $\xi \in \cX$, $\cF'(\xi) \subset \cF(\xi)$. 
\end{defn}

A natural class of outer approximations arises as follows.
\begin{defn}\label{defn:min_mv_map}
Given $\rho \geq 0$, the \emph{$\rho$-minimal combinatorial outer approximation} $\cF_\rho: \cX \mvmap \cX$ of $f$ on $\cX$ is defined inductively as follows.
If $\xi \in \cX^\text{top}$, then 
\[
\cF_\rho(\xi):=\setdef{\eta \in \cX}{B_{\rho}(f(|\xi|)) \cap |\eta| \neq \emptyset},
\]
where $B_{\rho}$ denotes the closed ball of radius $\rho$.
For all other $\xi\in \cX$, set
\[
\cF_\rho(\xi):= \bigcup_{\text{$\xi$ a face of $\xi'$}}\cF_\rho(\xi').
\]
\end{defn}

\begin{proposition}\cite[Corollary 2.6]{kalies:mischaikow:vandervorst:05}
\label{prop:enc_min_is_outer}
A combinatorial multivalued map $\cF$ is an outer approximation of $f$ if and only if $\cF$ encloses the minimal combinatorial outer approximation $\cF_0$ of $f$.
\end{proposition}

In practice, to incorporate errors we assume that $\rho > 0$, in which case the $\rho$-minimal combinatorial outer approximation encloses the minimal combinatorial outer approximation of $f$. 
As a consequence any $\rho$-minimal combinatorial outer approximation of $f$ is also a combinatorial outer approximation of $f$.

The following result is a special case of \cite[Proposition 7.3]{kalies:mischaikow:vandervorst:05}.

\begin{proposition}
\label{prop:fwdinv_is_ablock}
Let $\cF\colon \cX\mvmap \cX$ be an outer approximation of $f$. If $\cN\in\sInvset^+(\cF)$, then $|\cN|\in\sABlock(f)$.
\end{proposition}

Proposition~\ref{prop:fwdinv_is_ablock} together with \cite[Theorem 1.2]{kalies:mischaikow:vandervorst:15}\footnote{The remarks of \cite[pp. 1179]{kalies:mischaikow:vandervorst:15} guarantee that the assumptions of the theorem can be satisfied.} implies the following result.
\begin{proposition}
\label{prop:AtoN}
Let $\sA \subset \sAtt(f)$ be a finite lattice of attractors for $f$. There exists a sublattice $\sN$ of $\sABlock(f)$ and a lattice isomorphism $k\colon \sA \to \sN$ such that $\omega(k(A),f)=A$ for all $A \in \sA$.
\end{proposition}

From Proposition~\ref{prop:IndexPair} and Proposition~\ref{prop:fwdinv_is_ablock} we obtain the following corollary.

\begin{corollary}
\label{cor:indexPair}
Let $\cF\colon \cX\mvmap \cX$ be an outer approximation of $f$.
If $\cN_1,\cN_0\in \sInvset^+(\cF)$ and $\cN_0\subset\cN_1$, then $(|\cN_1|,|\cN_0|)$ is an index pair for $f$.
\end{corollary}

The importance of Corollary~\ref{cor:indexPair} is that it indicates that using an outer approximation of $f$, it is possible to identify index pairs for $f$.
What remains to be discussed is the fact that under the following additional constraint a combinatorial outer approximation can be used to compute Conley indices.

\begin{defn}
\label{defn:inscribedMV}
Let $\cU$ be a good closed cover of $X$ and let $\cX$ be a simplicial cover of $X$.
Let $\cF\colon \cX\mvmap \cX$ be a combinatorial multivalued map.
We say that $\cF$ is \emph{inscribed} into  $\cU$ if for every $x\in X$ there exists $U\in\cU$ such that
\[
\bigcup_{\xi \ni x}|\cF(\xi)| \subset \Int(U).
\]
\end{defn}

\begin{rem}
\label{rem:ConG}    
Let $\cU$ be a good closed cover of $X$, let $\cX$ be a simplicial cover of $X$, and assume that $\cF\colon \cX\mvmap \cX$ is a combinatorial multivalued map  \emph{inscribed} into  $\cU$.
Assume that $\cN_0\subset \cN_1$ are subcomplexes of $\cX$ with the property that $\cF(\cN_i)\subset \cN_i$.
Then, as is shown in \cite{harker:kokubu:mischaikow:pilarczyk} minor modifications of the algorithms presented in \cite{harker:mischaikow:mrozek:nanda} give rise to a well defined homomorphism
\[
\cF_*\colon H_*(\cN_1,\cN_0) \to H_*(\cN_1,\cN_0).
\]
Furthermore, if $\cF$ is an outer approximation of $f$ and $(|\cN_1|,|\cN_0|)$ is an index pair for $f$, then $\cF_*\colon H_*(\cN_1,\cN_0) \to H_*(\cN_1,\cN_0)$ is shift equivalent to $f_*\colon H_*(|\cN_1|,|\cN_0|) \to H_*(|\cN_1|,|\cN_0|)$.

An important consequence is as follows.
Let $\pi\colon \cX \to \bar{\cX}$ denote the quotient map that reduces $\cF$ to its condensation graph $\bar{\cF}$.
Let $q\in \sMG(\cF)$.
Set 
\begin{equation}
\cN(q) := \setdef{\xi\in\cX}{\text{$\exists \xi'\walk\xi$ under $\cF$ for some $\xi'\in \pi^{-1}(q)$}}\in\sInvset^+(\cF).
\end{equation}
We leave it to the reader to check that $\cN(q)\in \sJ^\vee(\sInvset^+(\cF))$ (see \cite{kalies:mischaikow:vandervorst:15}).
Thus, $(|\cN(q)|,|\cN(q)^<|)$ is an index pair and the Conley index of $\Inv(|\cN(q)\setminus\cN(q)^<|,f)$ is given by the shift equivalence class of $\cF_*\colon H_*(\cN(q),\cN(q)^<) \to H_*(\cN(q),\cN(q)^<)$.
\end{rem}

Remark~\ref{rem:ConG} motivates the following definition.

\begin{defn}
    \label{defn:ConGq}
Let $\cU$ be a good closed cover of $X$, let $\cX$ be a simplicial cover of $X$, and assume that $\cF\colon \cX\mvmap \cX$ is a combinatorial multivalued map  \emph{inscribed} into  $\cU$.
Let $q\in \sMG(\cF)$.
The \emph{Conley index} of $q$, denoted by $\Con(q)$, is defined to be the shift equivalence class of $\cF_*\colon H_*(\cN(q),\cN(q)^<) \to H_*(\cN(q),\cN(q)^<)$.
\end{defn}

\subsection{Approximating attracting blocks}
\label{sec:PiecewiseLinearApproximations}
Our goal is to prove the existence of an outer approximation of $f\colon X\to X$ from which we can identify the gradient-like and recurrent structures of $f$.
To do this requires that we have an appropriate approximation of $f$ and an appropriate representation of $X$.
We begin by considering the latter.

Throughout this section, $G$ denotes a function satisfying {\bf P1} - {\bf P3} and $\cX$ is a simplicial complex satisfying {\bf X1} and {\bf X2}.
For the remainder of this paper we use $\cG_\rho\colon \cX\mvmap \cX$ to denote the $\rho$-minimal combinatorial outer approximation of $G$ on $\cX$.
Note that {\bf P2} and {\bf P3} allow for (at least conceptually) the computation of $\cG_\rho$. 

We remark that if $G$ satisfies Theorem~\ref{thm:polyhedral-decomposition} then $\cX$ can be derived from $\cP$ by subdividing each polytope into a finite set of simplicies which in turn simplifies the evaluation of $\cG_\rho$.

In order to use $G$ to identify lattices of attracting blocks for $f$, we need control on the outer approximation generated by $G$.

\begin{defn}
Let $f\colon X \to X$ be a continuous function on a compact subset of $\R^d$.
Let $N\in \sABlock(f)$.
The \emph{tolerance} of $N$ is defined to be
\[
\tau(N) := \sup\setdef{\nu >0}{B_\nu(f(N))\subset N}
\]
\end{defn}

The assumption that $f$ is continuous implies that given $N\in \sABlock(f)$ there exists $a(N)>0$ such that 
\begin{equation}
\label{eq:collar}
 f(B_{a(N)}(N)) \subset B_{\tau(N)/4}(f(N)) .  
\end{equation}

\begin{proposition}
\label{prop:Aapprox}
Let $f\colon \R^d \to \R^d$ be a continuous function on $X$ a compact triangulable subset of $\R^d$ such that $f(X)\subset \Int(X)$.
Let $N\in\sABlock(f)$.
Choose $\rho(N)$ such that 
\[
0 <\rho(N) < \min\setof{a(N), \tau(N)/4}.
\]

Consider $G\colon X\to X$ such that $\sup_{x\in X}\| f(x)-G(x)\| < \rho(N)$ and assume $\diam(\cX) < \rho(N)$.
Let $s(N) := \setdef{\xi\in \cX}{|\xi|\cap N\neq\emptyset}$ and let\\
$\bdy_\cX(N):= \setdef{\xi\in \cX}{|\xi|\cap \bdy(N)\neq\emptyset}$.
Then
\begin{enumerate}
    \item [(i)] $s(N)\in\sInvset^+(\cG_{\rho(N)})$,
    \item [(ii)] $f(|s(N)|) \subset N \subset \Int(|s(N)|)$ and thus $|s(N)|\in\sABlock(f)$, and
    \item [(iii)] $\omega(|s(N)|,f) = \omega(N,f)$. 
\end{enumerate}
\end{proposition}

\begin{proof}
(i) Recall that $\cG_{\rho(N)}$ approximates $G$.
Since $\sup_{x\in X}\| f(x)-G(x)\| < \rho(N)$ and $\diam(\cX) < \rho(N)$,
\[
\cG_{\rho(N)}(s(N)) \subset B_{2\rho(N)}(G(B_{\rho(N)}(N))) \subset B_{3\rho(N)}(f(B_{\rho(N)}(N))).
\]
By definition of $\rho(N)$,
\[
B_{3\rho(N)}(f(B_{\rho(N)}(N))) \subset B_{3\rho(N)}(B_{\tau(N)/4}(f(N))) \subsetneq B_{\tau(N)}(f(N)).
\]
Therefore,
\[
\cG_{\rho(N)}(s(N)) \subset s(N)\setminus \bdy_\cX(N) \subset \Int(N) \subset s(N),
\] 
which implies that $s(N)\in\sInvset^+(\cG_{\rho(N)})$.

(ii) By definition of $\rho(N)$, $|s(N)| \subset B_{a(N)}(N)$ and thus by \eqref{eq:collar},
\[
f(|s(N)|) \subsetneq f(B_{a(N)}(N)) \subset B_{\tau(N)/4}(f(N))\subset N\subset \Int(|s(N)|).
\]
Therefore, $|s(N)|\in\sABlock(f)$.

(iii) By (ii), $f(|s(N)|)\subset N\subset |s(N)|$.
Thus, $\omega(f(|s(N)|),f)\subset \omega(N,f)\subset \omega(|s(N)|,f)$.
But $\omega(f(|s(N)|,f)= \omega(|s(N)|,f)$ and hence $\omega(N,f)=\omega(|s(N)|,f)$.
\end{proof}

\begin{rem}
The assumption $\diam(\cX)<\rho(N)$ can always be satisfied by taking barycentric subdivisions of the original simplicial complex over which $G$ is defined \cite[pp.~120]{Hatcher}.
\end{rem}

\begin{proposition}
\label{prop:epsilon}
Let $\epsilon > 0$.
Then, there exists $n = n(\epsilon)$ such that
\[
\max_{\xi\in \text{sd}^n(\cX)} \diam(G(|\xi|)) < \epsilon.
\]
\end{proposition}

\begin{proof}
Let $L$ denote the Lipschitz constant of $G$. Then\\
$
\max_{\xi\in \text{sd}^n(\cX)} \diam(G(|\xi|)) \leq L \diam(\text{sd}^n(\cX))$.
By \cite[pp.~120]{Hatcher},
\[
\diam(\text{sd}^n(\cX)) \leq \left( \frac{d}{d+1} \right)^n \diam(\cX).
\]
Choose $n = n(\epsilon)$ such that
\[
\left( \frac{d}{d+1} \right)^n < \frac{\epsilon}{L\diam(\cX)}.
\]
\end{proof}

\begin{proposition}
\label{prop:inscribed}
Let $\cU$ be a good closed cover of $X$.
Set $\delta(\cU) > 0$ equal to a Lebesgue number of the cover $\cU^o$.
Suppose $0 < \rho < \delta(\cU)$.
Choose $n = n(\delta(\cU) - \rho)$ as in Proposition~\ref{prop:epsilon}.
Then, $\cG_{\rho}: \text{sd}^n(\cX) \mvmap \text{sd}^n(\cX)$ is inscribed into $\cU$.
\end{proposition}

\begin{proof}
Let $x \in X$.
By definition of a Lebesgue number, there exists $U \in \cU$ such that $B_{\delta(\cU)}(G(x)) \subset \Int(U)$.
By Proposition \ref{prop:epsilon},
$\max_{\xi \in \text{sd}^n(\cX)}\diam(G(\xi)) < \delta(\cU) - \rho$.
This implies
\[
\bigcup_{\xi \ni x}\cG_{\rho}(\xi) \subset B_{\rho}(B_{\delta(\cU) - \rho}(G(x))) = B_{\delta(\cU)}(G(x)) \subset \Int(U),
\]
and thus $\cG$ on $\text{sd}^n(\cX)$ is inscribed into $\cU$.
\end{proof}

\section{Proof of Theorem~\ref{thm:intro:main}}
\label{sec:proof}

We make use of Figure~\ref{fig:G} to illuminate the key challenges in the proof of Theorem~\ref{thm:intro:main}.
 
From the hypotheses of Theorem~\ref{thm:intro:main}, we have the existence of a finite union of closed, bounded, convex $d$-dimensional polytopes $X\subset \R^d$,
continuous function $f\colon \R^d\to \R^d$ such that $f(X)\subset \Int(X)$, and $(\sMR,\leq_{\sMR_f})$, a Morse representation of $f$.

Figure~\ref{fig:G}(A) provides an example of a Morse representation $\sMR$ that consists of three Morse sets $M_0$, $M_1$, and $M_2$.
For the sake of simplicity, we will assume that $M_0$ and $M_2$ are fixed points.
The Morse set $M_1$ consists of an interval.
By definition $M_1$ is an invariant set, but we make no further claims concerning the structure of the dynamics on $M_1$.

\begin{figure}
\centering
\begin{tikzpicture}[scale=0.8]
  \draw[] (-0.3,0) -- (16.3,0); 
  \draw[thick] (0.3,0) -- (15.7,0); 
  \draw[blue,fill=blue] (5/4,0) circle (.5ex) node [below] {$M_0$};
  \draw[blue,fill=blue] (21/4,0) circle (.5ex) node [below] {$M_2$};
  \draw[blue, ultra thick] (8.75,0) -- (61/4,0); 
  \draw[blue] (49/4,-10pt) node  {$M_1$};
  \draw[] (-0.5,-10pt) node  {\textbf{(a)}};
\end{tikzpicture}

\vspace{0.25cm}
\begin{tikzpicture}[scale=0.8]
  \draw[] (-0.3,0) -- (16.3,0); 
  \draw[thick] (0.3,0) -- (15.7,0); 
\draw[red, ultra thick] (0.3,0) -- (2.8,0) node [below] {$N_0$};
\draw[red, ultra thick] (7.2,0) -- (15.7,0) node [below] {$N_1$};
  \draw[blue,fill=blue] (5/4,0) circle (.5ex) node [below] {$M_0$};
  \draw[blue,fill=blue] (21/4,0) circle (.5ex) node [below] {$M_2$};
  \draw[blue,  thick] (8.75,0) -- (61/4,0); 
  \draw[blue] (49/4,-10pt) node  {$M_1$};
  \draw[] (-0.5,-10pt) node  {\textbf{(b)}};
\end{tikzpicture}

\vspace{0.25cm}
\begin{tikzpicture}[scale=0.8]
  \draw[] (-0.3,0) -- (16.3,0);
  \foreach \x in  {0,1,...,32}
    \draw (\x/2,2pt) -- +(0,-4pt);
\draw[purple, ultra thick] (0,0) -- (3,0) node [below] {$s(N_0)$};
\draw[purple, ultra thick] (7,0) -- (16,0) node [below] {$s(N_1)$};
  \draw[] (0.0,-10pt) node  {\textbf{(c)}};
\end{tikzpicture}

\vspace{0.2cm}
\begin{tikzpicture}[scale=0.8]
  \draw[] (-0.3,0) -- (16.3,0); 
  \foreach \x in  {0,1,...,32} 
    \draw (\x/2,2pt) -- +(0,-4pt);
\draw[blue, ultra thick] (1,0) -- (1.5,0) node [below] {$q_0$};
\draw[blue, ultra thick] (2,0) -- (2.5,0) node [below] {$q_4$};
\draw[blue, ultra thick] (3.5,0) -- (4,0) node [below] {$q_5$};
\draw[blue, ultra thick] (5,0) -- (6,0) node [below] {$q_7$};
\draw[blue, ultra thick] (7.5,0) -- (8,0) node [below] {$q_6$};
\draw[blue, ultra thick] (8.5,0) -- (9.5,0) node [below] {$q_1$};
\draw[blue, ultra thick] (10.5,0) -- (11,0) node [below] {$q_3$};
\draw[blue, ultra thick] (12.5,0) -- (15.5,0) node [below] {$q_2$};
  \draw[] (-0.5,-10pt) node  {\textbf{(d)}};
\end{tikzpicture}
    \caption{(a)~$X$ is denoted by the thick black line. Morse representation $\sMR = \setof{M_0,M_1,M_2}$ with partial order $M_0 <_{\sMR_f} M_2$ and $M_1 <_{\sMR_f} M_2$.
    (b)~$N_0 = k(\psi(M_0))$, $N_1 = k(\psi(M_1))$ and $N_2 = k(\psi(M_2)) = X$
    (not shown) are join irreducible attracting blocks.
    (c)~Hash marks represent the grid $\cX$ with diameter less than $\rho(\sN)$ that is used to define $s(N_0)$ and  $s(N_1)$.
    (d)~The Morse graph for $\cG$ is $\sMG(\cG) = \setdef{q_i}{i=0,\ldots, 7}$ with partial order $q_0 <_{\sMG(\cG)} q_4 <_{\sMG(\cG)} q_5 <_{\sMG(\cG)} q_7$, $q_1 <_{\sMG(\cG)} q_6 <_{\sMG(\cG)} q_7$, $q_1<_{\sMG(\cG)} q_3$, and $q_2<_{\sMG(\cG)}q_3$.}
    \label{fig:G}
\end{figure}

By Remark~\ref{rem:MRtoAtt} there exists a unique finite lattice of attractors $\sA$ for which there is a poset isomorphism $\psi\colon \sMR \to \sJ^\vee(\sA)$.

By Proposition \ref{prop:AtoN} there exists a sublattice $\sN$ of $\sABlock(f)$ (with minimal element $\bzero_\sN := \emptyset$ and maximal element $\bone_\sN = X$) and a lattice isomorphism $k\colon \sA \to \sN$ such that $\omega(k(A),f)=A$ for all $A \in \sA$.
Figure~\ref{fig:G}(B) shows a choice of attracting blocks $N_i = k(\psi(M_i))$ for $i=0,1$. By \cite[Theorem 1.2]{kalies:mischaikow:vandervorst:15}, $k(\psi(M_2)) = X$.

While \cite[Theorem 1.2]{kalies:mischaikow:vandervorst:15} provides for the existence of a sublattice $\sN$, there is no claim that it is unique (see Example~\ref{ex:Ablock}).
Therefore, we emphasize that for the remainder of this section we assume that $\sN$ is fixed  with minimal element $\bzero_\sN := \emptyset$ and maximal element $\bone_\sN = X$.
For each $N\in\sN$, choose $\rho(N) > 0$ satisfying Proposition~\ref{prop:Aapprox}.
Set
\[
\rho(\sN) = \min \setdef{\rho(N)}{N\in\sN}.
\]
For the remainder of this section, we consider $G\colon X\to X$ 
satisfying {\bf P1} where $\diam(\cX) < \rho(\sN)$ and set $\cG = \cG_{\rho(\sN)}$.

Given $N\in\sN$, define
$s(N):= \setdef{\xi\in\cX}{\cl(\xi)\cap N\neq\emptyset}$
and $\sK := \setdef{s(N)}{N\in\sN}$.
The intervals defined by the hash marks in Figure~\ref{fig:G}(C) are the polytopes for the complex $\cX$ with $\diam(\cX)<\rho(\sN)$.
The sets $s(N_i)$, for $i=0,1$ are also shown.

\begin{proposition}
\label{prop:NisoK}
    The map $s\colon \sN \to \sK$ is a lattice isomorphism where $\sK$ is a sublattice of $\sInvset^+(\cG)$.
\end{proposition}

\begin{proof}
By Proposition~\ref{prop:Aapprox} part~(i), $\sK\subset \sInvset^+(\cG)$.
Let $N_0, N_1 \in \sN$. If $N_0 \cap N_1 \neq \emptyset$, then $s(N_0 \cap N_1) = s(N_0) \cap s(N_1)$. Suppose $N_0 \cap N_1 = \emptyset$.
By Proposition~\ref{prop:Aapprox} part~(ii), $f(|s(N_i)|) \subset N_i$ for $i = 0, 1$. Hence, $f(|s(N_0) \cap s(N_1)|) \subset N_0 \cap N_1$, so $s(N_0) \cap s(N_1)= \emptyset$.

We leave it to the reader to check that $s(\emptyset)=\emptyset$, $s(X)=\cX$, and $s(N_0\cup N_1) = s(N_0)\cup s(N_1)$.

It remains to show that $s$ is an isomorphism.
By definition $s$ is surjective.
To show that $s$ is injective, suppose that $N_0,N_1\in \sN$ and $s(N_0) = s(N_1)$. 
Recall that we have fixed an isomorphism $k\colon \sA \to \sN$. Let $A_i := k^{-1}(N_i)$.
By Proposition~\ref{prop:Aapprox} part~(iii), $\omega(|s(N_i)|,f) = \omega(N_i,f)$. 
Hence 
\[
A_0= \omega(k(A_0),f) = \omega(N_0,f)=\omega(N_1,f) = \omega(k(A_1),f) = A_1.
\]
Since $k$ is an isomorphism, $N_0 = N_1$.
\end{proof}

\begin{proposition}
\label{prop:Mcontained}
If $M\in \sMR$, then $M\subset |s(k(\psi(M)))\setminus s(k(\psi(M)^<))|$.
\end{proposition}
\begin{proof}
Observe that $M\subset \psi(M)\subset k(\psi(M))\subset |s(k(\psi(M)))|$.
Thus it is sufficient to prove that $M\cap |s(k(\psi(M)^<))| = \emptyset$.

By Remark~\ref{rem:MRtoAtt},
\[M \cap \psi(M)^< = \psi(M)\cap ((\psi(M)^<)^* \cap \psi(M)^<) = \psi(M) \cap \emptyset = \emptyset.
\]

We now show that $M\cap k(\psi(M)^<)= \emptyset$.
We leave it to the reader to check that given $U, V \subset X$, $\omega(U \cap V, f) \subset \omega(U, f) \cap \omega(V, f)$. 
Thus
\[
\omega(M\cap k(\psi(M)^<),f) \subset \omega(M,f) \cap \omega(k(\psi(M)^<),f),
\]
and
\[
\omega(M,f) \cap \omega(k(\psi(M)^<),f) = M \cap \psi(M)^< = \emptyset.
\]
Since $\omega(M\cap k(\psi(M)^<), f) = \emptyset$, $M\cap k(\psi(M)^<) = \emptyset$.

By Proposition \ref{prop:Aapprox} part~(ii), 
$f(M\cap |s(k(\psi(M)^<))|) \subset M \cap k(\psi(M)^<)=\emptyset$. Thus, $M\cap |s(k(\psi(M)^<))|=\emptyset$.
\end{proof}

The reader can check that the illustration of Figure~\ref{fig:G}(C)
is compatible with Proposition~\ref{prop:Mcontained}.

Let $\pi\colon \cX \to \bar{\cX}$ denote the quotient map that reduces $\cG$ to its condensation graph $\bar{\cG}$.
The intervals highlighted in Figure~\ref{fig:G}(D) are the Morse nodes of the Morse graph $\sMG(\cG) = \setdef{q_j}{j=0,\ldots,7}$ of the outer approximation $\cG$ of $G$.
We use this hypothetical example to highlight two challenges that arise from using the dynamics computed via $\cG$ to describe the dynamics of $f$: spurious recurrence and overly refined recurrence.

Before turning to the challenges we note that the Morse nodes $q_0$ and $q_7$ identify the location of the Morse sets $M_0$ and $M_2$ of $f$, respectively, in that $M_0\subset |\pi^{-1}(q_0)|$ and $M_2\subset |\pi^{-1}(q_7)|$.
Furthermore, as is shown in the proof of Theorem~\ref{thm:intro:main}, $\Con(M_0)\sim \Con(q_0)$ and $\Con(M_2)\sim \Con(q_7)$.
This represents the ideal information that can be provided by $\cG$.

With respect to the dynamics of $f$, the Morse nodes $q_4$, $q_5$, and $q_6$ are spurious, meaning that they identify possible recurrent dynamics that does not exist.
An explanation for the existence of these Morse nodes is that $G$ does not provide a sufficiently good approximation of $f$ to identify that $f$ does not exhibit recurrent dynamics at these locations of phase space. 
However, by Definition~\ref{defn:ConGq} and Remark~\ref{rem:ConG} $\Con(q_i)=0$ for $i=4,5,6$, otherwise $\Inv((N(q)\setminus N(q)^<,f)\neq \emptyset$.
Note that the user is thus alerted of the potential that  spurious invariant sets have been identified (as indicated in Example~\ref{ex:ConleyIndex2}, trivial Conley index does not imply that the associated invariant set is trivial).

Finally, the Morse nodes $q_1$, $q_3$ and $q_2$ are indicative of overly refined recurrence.
These Morse nodes intersect with, but do not cover the Morse set $M_1$.
Recall that the hypothesis of Theorem~\ref{thm:intro:main} and the information provided in Figure~\ref{fig:G}(A) is the choice of a Morse representation $\sMR$ of $f$.
Overly refined recurrence arises in regions where $G$ tightly approximates $f$, which allows  $\cG$ to identify finer dynamics than that captured by $\sMR$.
In particular, with respect to the example of Figure~\ref{fig:G} we can conclude that on the region of phase space $M_1$ the dynamics of $f$ can be further decomposed into at least three recurrent sets that lie in the regions $|\pi^{-1}(q_1)|$, $|\pi^{-1}(q_3)|$ and $|\pi^{-1}(q_2)|$, with connecting orbits from $q_3$ to $q_1$ and $q_3$ to $q_1$.

Overly refined recurrence occurs if there exists a Morse representation of $f$ that is finer than $\sMR$, and $\cG$ identifies this finer Morse representation.
Therefore, to satisfy the conclusion of Theorem~\ref{thm:intro:main} we collapse this information to construct the desired Morse representation.

Figure~\ref{fig:G}(D) illustrates the following two propositions.

\begin{proposition}
\label{prop:qinM}
If $M\in \sMR$, then there exists $q\in \sMG(\cG)$ such that  $M\cap |\pi^{-1}(q)|\neq\emptyset$.
\end{proposition}

\begin{proof}
Let $x_0\in M$. 
Since $M$ is invariant, its forward trajectory\\ $\setdef{x_{n+1} := f(x_n)}{n\in\N}\subset M$.  
For each $n\in \N$, choose $\xi_n\in \cX$ such that $x_n\in\xi_n$.
Since $\cG$ is an outer approximation of $f$, $\xi_{n+1}\in \cG(\xi_n)$, i.e., $\setof{\xi_n}$ defines an infinite walk in $\cX$.
Since $\cX$ is finite, $\setof{\xi_n}$ must contain a cycle and thus intersects at least one recurrent component, $\pi^{-1}(q)$ for some $q\in \sMG(\cG)$.
The assumption that $x_n\in \xi_n$ implies that $M\cap |\pi^{-1}(q)|\neq\emptyset$.
\end{proof}

\begin{lemma}
\label{lem:recurrent_in_fwd_inv}
Let $q_1 \in \sMG(\cG)$ and $A \in \psi(\sMR)$.
If $\pi^{-1}(q_1) \cap s(k(A)) \neq \emptyset$, then $\pi^{-1}(q_1) \subset s(k(A))$.
Furthermore, if $q_0 \leq_{\sMG(\cG)} q_1$, then $\pi^{-1}(q_0) \subset s(k(A))$.
\end{lemma}

\begin{proof}
By Proposition~\ref{prop:Aapprox} part~(i), $s(k(A))\in \sInvset^+(\cG)$. Since $\pi^{-1}(q_1)$ is a recurrent component of $\cG$, then $\pi^{-1}(q_1) \cap s(k(A)) \neq \emptyset$ implies $\pi^{-1}(q_1) \subset s(k(A))$.

By definition, the assumption that $q_0 \leq_{\sMG(\cG)} q_1$ implies that there exists a walk $\xi_1\walk \xi_0$ in $\cG$ such that $\xi_0 \in \pi^{-1}(q_0)$ and $\xi_1 \in \pi^{-1}(q_1)$.
Since $s(k(A))\in \sInvset^+(\cG)$, the walk $\xi_1\walk\xi_0$ is contained in $s(k(A))$.
Thus, $\pi^{-1}(q_0) \subset s(k(A))$.
\end{proof}

\begin{proposition}
\label{prop:interval}
Given $M\in \sMR$, set 
\[
I(M):= \setdef{q\in \sMG(\cG)}{ \pi^{-1}(q)\subset s(k(\psi(M)))\setminus s(k(\psi(M)^<))}.
\]
Then,  $I(M)$ is a nonempty convex subset of $\sMG(\cG)$.    
\end{proposition}

\begin{proof}
By Proposition~\ref{prop:qinM}, there exists $q\in \sMG(\cG)$ such that  $M\cap |\pi^{-1}(q)|\neq\emptyset$.
By Proposition~\ref{prop:Mcontained}, $M\subset |s(k(\psi(M)))\setminus s(k(\psi(M)^<))|$.
This implies that $\pi^{-1}(q)\cap s(k(\psi(M))) \neq\emptyset$.
By Lemma~\ref{lem:recurrent_in_fwd_inv}, $\pi^{-1}(q)\subset s(k(\psi(M)))$.
Then\\
$\pi^{-1}(q)\cap s(k(\psi(M)^<)) = \emptyset$, for otherwise Lemma~\ref{lem:recurrent_in_fwd_inv} implies that\\ $\pi^{-1}(q)\subset s(k(\psi(M)^<))$, contradicting the fact that $M\cap |\pi^{-1}(q)|\neq\emptyset$ 
and $M\subset |s(k(\psi(M)))\setminus s(k(\psi(M)^<))|$.
Therefore, $I(M)\neq\emptyset$.

If $I(M)$ consists of a unique element, then the result is trivially true.
Thus, we need to show that if $q_0,q_1\in I(M)$ and $q_0<_{\sMG(\cG)}  q' <_{\sMG(\cG)} q_1$, then $q'\in I(M)$.
We provide a proof by contradiction, i.e., we assume that $q'\not\in I(M)$.

By Lemma~\ref{lem:recurrent_in_fwd_inv}, $\pi^{-1}(q') \subset s(k(\psi(M)))$. By assumption $q'\not\in I(M)$, so $\pi^{-1}(q') \subset s(k(\psi(M)^<)))$. Then by Lemma~\ref{lem:recurrent_in_fwd_inv}, $\pi^{-1}(q_0) \subset s(k(\psi(M)^<)))$. Therefore, $q_0\not\in I(M)$, a contradiction.
\end{proof}

\begin{proposition}
\label{prop:partition}
$\setdef{I(M)}{M\in\sMR}$ is a partition of $\sMG(\cG)$.
\end{proposition}
\begin{proof}
Recall that $\psi: \sMR \to \sJ^\vee(\sA)$ is a poset isomorphism and $k(\sJ^\vee(\sA)) = \sJ^\vee(\sN)$.
By \cite[Proposition 2.45]{davey:priestley}, $X = \bone_\sN 
 = \bigvee_{M\in \sMR}k(\psi(M)) = \bigcup_{M\in \sMR}k(\psi(M))$.
Therefore, $\cX = \bigcup_{M\in \sMR}s(k(\psi(M)))$.

Let $q\in \sMG(\cG)$. Then $\pi^{-1}(q) \cap s(k(\psi(M_0)))\neq \emptyset$ for some $M_0\in \sMR$.

By Lemma~\ref{lem:recurrent_in_fwd_inv},
\[
\pi^{-1}(q) \subset s(k(\psi(M_0)))\setminus s(k(\psi(M_0)^<))\quad\text{or}\quad\pi^{-1}(q) \subset  s(k(\psi(M_0)^<)).
\]

If $\pi^{-1}(q) \subset s(k(\psi(M_0)))\setminus s(k(\psi(M_0)^<))$ then we are done.
So assume that $\pi^{-1}(q) \subset  s(k(\psi(M_0)^<)) \neq \emptyset$.
If $\psi(M_0)^< \in \sJ(\sA)$, then we repeat the argument, i.e.,
\[
\pi^{-1}(q) \subset s(k(\psi(M_0)^<))\setminus s(k((\psi(M_0)^<)^<))\quad\text{or}\quad\pi^{-1}(q) \subset  s(k((\psi(M_0)^<)^<)).
\]
So assume $\psi(M_0)^< \not\in \sJ(\sA)$.
As given by
Proposition~\ref{prop:NisoK} and \cite[Theorem 1.2]{kalies:mischaikow:vandervorst:15},
$s$ and $k$ are isomorphisms, which implies $\psi(M_0)^< \neq \emptyset$.
Then by \cite[Proposition 2.45]{davey:priestley}, $\psi(M_0)^<$ can be written as the non-empty join $\psi(M_0)^<=\bigvee_{M<M_0}s(k(\psi(M)))$.
Thus, $\pi^{-1}(q) \cap s(k(\psi(M_1)))\neq \emptyset$ for some $M_1 < M_0$ and we repeat the argument.

These arguments can only be repeated finitely many times and hence there exists $M\in \sMR$ such that $q\in I(M)$. By definition, $q$ belongs to a unique $I(M)$, and therefore, $\setdef{I(M)}{M\in\sMR}$ is a partition of $\sMG(\cG)$.
\end{proof}

Part~(i) of Theorem~\ref{thm:intro:main}, which may be of independent interest, follows from the  following proposition. 

\begin{proposition}
The map $\nu\colon \sMG \to \sMR$ defined by $\nu(q) = M$ if $q\in I(M)$ is a poset epimorphism.    
\end{proposition}
\begin{proof}
Proposition~\ref{prop:partition} implies that $\nu$ is well-defined, and Proposition~\ref{prop:interval} implies that $\nu$ is surjective.

To show that $\nu$ is a poset morphism, we need to show that if $q_0 \in I(M_0)$, $q_1 \in I(M_1)$, and $q_0 \leq_{\sMG(\cG)} q_1$, then $M_0 \leq_{\sMR_f} M_1$. We argue by contradiction, i.e. we assume that $M_1 <_{\sMR} M_0$. Since $\psi$, $k$, and $s$ are order-preserving, this assumption implies that $s(k(\psi(M_1))) <_{\sInvset^+(\cG)} s(k(\psi(M_0)))$, or equivalently $s(k(\psi(M_1))) \subsetneq s(k(\psi(M_0)))$.
Then $s(k(\psi(M_1))) \subset s(k(\psi(M_0)))^< = s(k(\psi(M_0)^<))$.

By definition of $I(M_1)$, $\pi^{-1}(q_1) \subset s(k(\psi(M_1)))$.
Since $q_0 \leq_{\sMG(\cG)} q_1$, Lemma~\ref{lem:recurrent_in_fwd_inv} implies that 
$\pi^{-1}(q_0) \subset s(k(\psi(M_1)))$.
Then $\pi^{-1}(q_0) \subset s(k(\psi(M_0)^<))$, contradicting the assumption that $q_0 \in I(M_0)$.
\end{proof}

Part~(ii) of Theorem~\ref{thm:intro:main} follows from Proposition \ref{prop:enc_min_is_outer}, Remark \ref{rem:ConG}, and Proposition~\ref{prop:inscribed}.

\section{Examples}
\label{sec:examples}

In this section, to discuss the applicability of Theorem~\ref{thm:intro:main}, we generate synthetic data using a two-dimensional nonlinear population model, 
\begin{equation}
\label{eq:Leslie}
f_\theta(x):=
\begin{bmatrix}
    (\theta_1 x_1+\theta_2 x_2)e^{-0.1(x_1+x_2)}\\ 0.7x_1
\end{bmatrix}
\end{equation}
where the parameter $\theta$ is chosen to be $\btheta = (23.5,23.5)$.
As this is meant to be a population model it makes sense to restrict our attention to data in the non-negative orthant $[0,\infty)^2$ and we leave it to the reader to check that for $X = [0, 90] \times [0, 70]$, $f_\btheta(X) \subset \Int(X)$.

We choose \eqref{eq:Leslie} for four reasons.
First, mathematically rigorous computations of the dynamics generated by $f_\theta$ can be performed (see \cite{arai:kalies:kokubu:mischaikow:oka:pilarczyk,bush:gameiro:harker:kokubu:mischaikow:obayashi:pilarczyk}).
Second, it is known that at $\btheta$, the dynamics is complicated (see \cite[Figure 1]{ugarcovici:weiss}).
Third,  as a function of $\theta$ the dynamics of $f_\theta$ undergoes dramatic changes over complicated sets of bifurcation points (see \cite[Section 5]{ugarcovici:weiss}).
Fourth, we can demonstrate how, with relatively few data points, we can capture interesting dynamics, and indicate where failure is to be expected. 

\subsection{Baseline}
To obtain a baseline, we use the software package \cite{CMGDB} to construct multivalued maps by evaluating $f_\btheta$ on the vertices of a cubical cell complex obtained by uniformly subdividing $X = [0, 90] \times [0, 70]$. 
We consider two levels of subdivision, and therefore two distinct cell complexes -- $\cZ_1$ and $\cZ_2$ consisting of $2^{18}$  and $2^{27}$ rectangles, respectively -- and corresponding multivalued maps $\cF_1$ and $\cF_2$.
Figure~\ref{fig:baseline} displays the Hasse diagrams of the Morse graphs together with the Conley indices and the corresponding color-coded subsets of phase space. 
Comparing Figures~\ref{fig:baseline}(a) and (b), we observe that the greater level of subdivision leads to a more refined description of the dynamics.

Each node $q_k$ of the Morse graph has a label of the form $k : (s_0(x),s_1(x),s_2(x))$ where $k$ identifies the node and $s_i(x)$ is the product of the invariant factors (excluding the nilpotent blocks) \cite[Section 7.2]{hoffman:kunze} of the $i$-th homology map of the index pair (see Definition~\ref{defn:ConleyIndex}).
Since $X\subset \R^2$, $s_i(x) \equiv 0$ for $i\geq 3$.
We use $\F_5$ coefficients for the computations presented here.

Turning to Figure~\ref{fig:baseline}(a), the Morse graph $\sMG(\cF_1)$ has four nodes $\setof{0,1,2,3}$.
The Conley index of node $0$, $(x^3-1,0,0)$, is that of a stable period three orbit.
From this we can conclude that $\Inv(|\pi^{-1}(0)|,f_\btheta)$ contains a period three orbit (though it need not be stable).
Similarly, the Conley index of node $1$ is that of an unstable (orientation preserving) period-3 orbit and hence $\Inv(|\pi^{-1}(1)|,f_\btheta)$ contains a period three orbit.
Finally, the Conley index of node $3$ is that of a two-dimensional unstable (orientation preserving) fixed point and hence $\Inv(|\pi^{-1}(3)|,f_\btheta)$ contains a fixed point.
Note that in each case we obtained lower bounds on the dynamics of the invariant sets associated with the Morse nodes.
In fact, the dynamics is much more complicated and as indicated in the introduction this section, for some values of $\theta$, too complicated to be completely characterized by finite data (see \cite[Section 5]{ugarcovici:weiss}), while our characterization via Morse graphs and Conley indices can be made rigorous (see \cite{arai:kalies:kokubu:mischaikow:oka:pilarczyk}).

The Conley index of node 2 is trivial. 
However, $|\pi^{-1}(2)|$ contains the origin, which is an unstable fixed point for \eqref{eq:Leslie}.
Thus, $\Inv(|\pi^{-1}(2)|,f_\btheta)\neq \emptyset$.

\begin{figure}
    \centering
    \includegraphics[width=\linewidth]{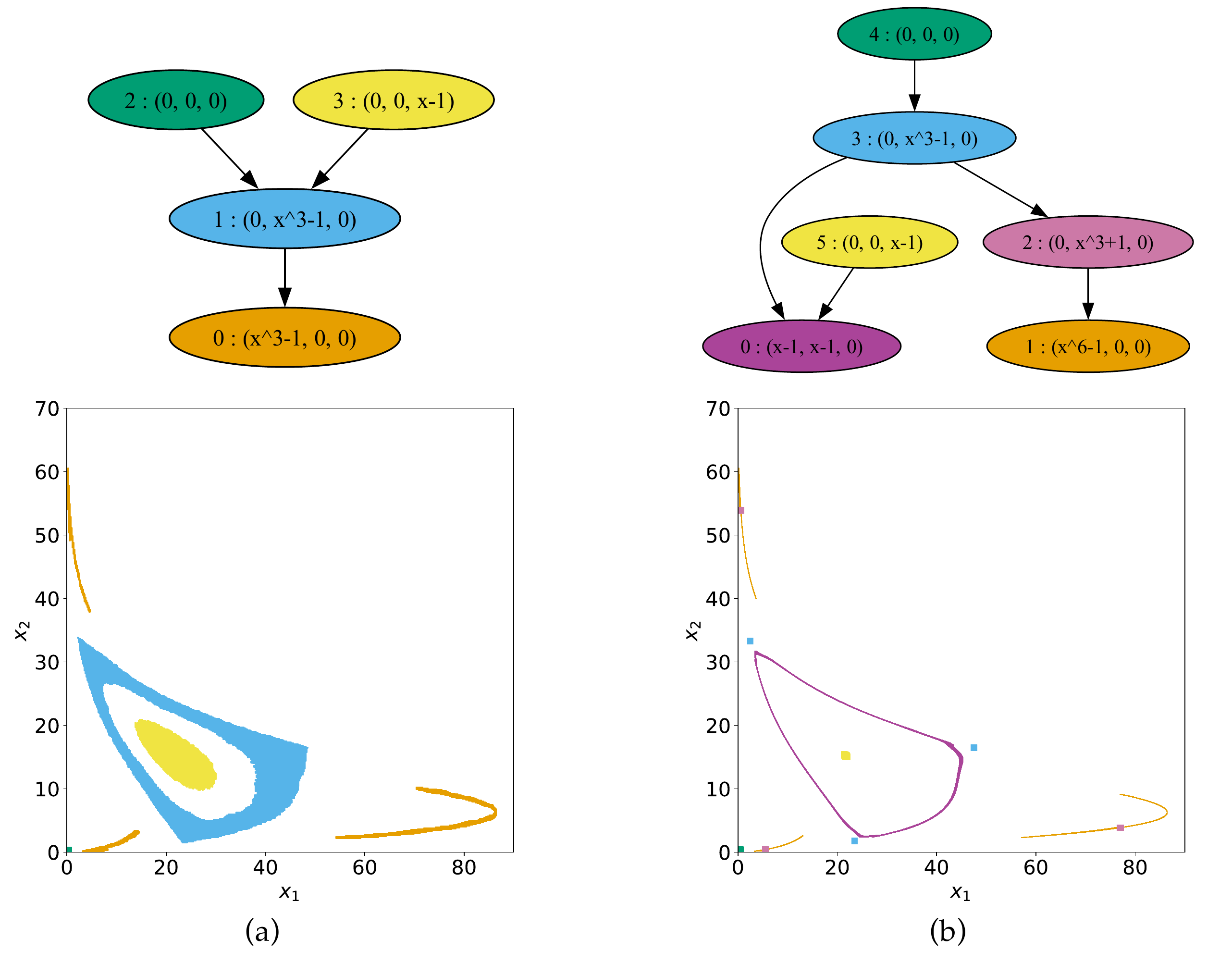}
    \caption{
    Combinatorial dynamics of $\cF$ where $\cF$ is a multivalued map computed directly using $f_\btheta$ given by \eqref{eq:Leslie}. We use these results as a baseline for comparison to the results in Figure~\ref{fig:ML_approx}.
    In both figures, the Hasse diagram of the Morse graph $\sMG(\cF)$ is displayed above a plot of phase space where the sets $|\pi^{-1}(q)|$ are plotted with colors corresponding to $q \in \sMG(\cF)$.
    The cubical cell complex $\cX$ used to obtain these results is constructed by uniformly subdividing phase space into (a) $2^{18}$ rectangles and (b) $2^{27}$ rectangles. For visualization purposes, in plot (a), $|q^{-1}(2)|$ is enlarged by a factor of 10. In plot (b), $|q^{-1}(k)|$ is enlarged by a factor of 100 for $k = 2, 3, 5$, and $|q^{-1}(4)|$ is enlarged by a factor of 200.}
    \label{fig:baseline}
\end{figure}

We remark that since the invariant sets associated with each node of the Morse graph are nontrivial, we have identified a Morse representation of $f_\btheta$, which we denote by $\sMR_1$.
As indicated in the introduction, the use of the language of Morse representations allows us to analyze dynamics at different levels of resolution.

Figure~\ref{fig:baseline}(b) provides a finer level of resolution.
The Morse graph has six nodes $\setof{0,1,2,3,4,5}$.
Nodes $0$ and $1$ are minimal indicating bistability.
The fundamental differences between $\sMG(\cF_1)$ and $\sMG(\cF_2)$ are that $\cF_2$ reveals that the Morse sets defined by $\Inv(|\pi^{-1}(1)|,f_\btheta)$ and $\Inv(|\pi^{-1}(0)|,f_\btheta)$ in Figure~\ref{fig:baseline}(a) can each be decomposed into two Morse sets.
In particular, in Figure~\ref{fig:baseline}(b), $\Inv(|\pi^{-1}(0)|,f_\btheta)$ has the Conley index of an invariant circle and $\Inv(|\pi^{-1}(3)|,f_\btheta)$ has the Conley index of an unstable (orientation preserving) period-3 orbit, while  
$\Inv(|\pi^{-1}(1)|,f_\btheta)$ has the Conley index of a stable  period-6 orbit and $\Inv(|\pi^{-1}(2)|,f_\btheta)$ has the Conley index of an unstable (orientation reversing) period-3 orbit.
We denote the Morse representation associated to this example by $\sMR_2$.

\subsection{Model and Training}
As proof of concept, we use a standard machine learning model to approximate $f$ and demonstrate Theorem~\ref{thm:intro:main}.
We construct data sets of the following form
\[
\cD(T,I) := \setdef{f^k_\btheta(x_i)}{0\leq k\leq T,\ x_i\in X,\ 1\leq i \leq I}
\]
where the initial conditions $x_i$ are chosen uniformly at random in $X$, and we train using the pairs $\setof{(f_\btheta^{k-1}(x_i),f_\btheta^k(x_i))}$.

Fix $G$ to be a fully-connected feedforward neural network with ReLU activation, depth three, and layers of identical width $w$. We train the parameters of $G$ using PyTorch \cite{pytorch} by minimizing the mean squared error with the Adam optimizer \cite{kingma:ba}. The learning rate is initialized at $0.0001$ and reduced on plateau.

Similarly to the baseline method, we uniformly subdivide phase space to obtain a cubical cell complex. Using \cite{CMGDB}, we construct a multivalued map $\cG$ by evaluating $G$ on the vertices of each rectangle.

\subsection{Results}
\label{sec:examples:results}
We present a few simple results, but encourage the interested reader to explore with additional experiments \cite{RCDML}.

\begin{figure}
    \centering
    \includegraphics[width=\linewidth]{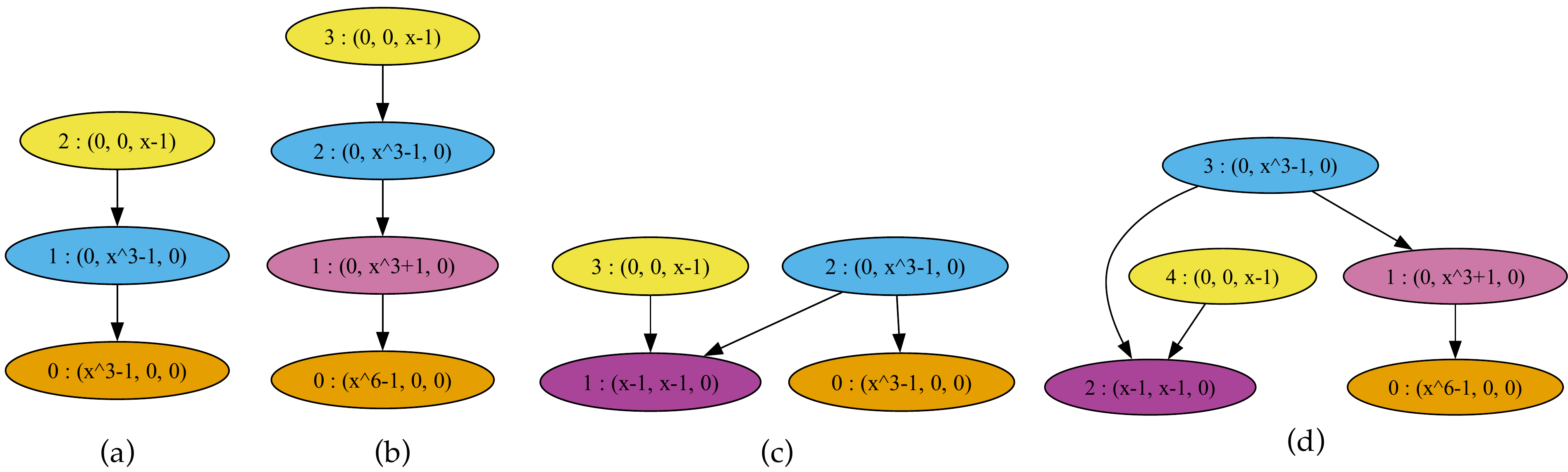}
    \caption{The Hasse diagrams of the Morse graphs
    computed using neural network approximations to $f_\theta$ as defined by Equation~\ref{eq:Leslie}.
    }
    \label{fig:ML_approx}
\end{figure}

\subsubsection{Experiment 1}

The results presented in Figure~\ref{fig:ML_approx}(a) were obtained using  $\cD(10,16)$, $w = 32$, and $2^{20}$ rectangles.
In this case $\nu\colon \sMG \to \sMR_1$ of Theorem~\ref{thm:intro:main} is an embedding, but not an epimorphism.
This result is easily reproducible: in an experiment with ten trials fixing $w$, $n$, and the number of rectangles, but varying the initial weights of the network at the start of training, we observe 
the same Morse graph in every trial.

What $\nu$ does not identify is the Morse set associated with the unstable fixed point at the origin.
However, this is to be expected. 
Naively, one expects that for $G$ to learn that there is recurrence near the origin requires the existence of a data point in the rectangle containing the origin whose image is also in that rectangle.
A simple computation shows that the value of the unstable eigenvalue at the origin is approximately 24 and thus the area of the region containing such data points is roughly $1.6 \times 10^{-7}$ the area of $X$.
Therefore, to reliably obtain such a data point suggests the need for on the order of $10^{10}$ data points.

Here we appeal to the aphorism ``perfect is the enemy of the good.'' Theorem~\ref{thm:intro:main} can be achieved but at an unreasonable cost compared to the fact that the  interesting complex dynamics can be obtained with relatively few data points.
It is also worth noting that the difficulty of identifying the fixed point at the origin is due to the fact that it is unstable \emph{and} lies on the boundary of $X$.
The Morse set associated with node 3 in Figure~\ref{fig:baseline}(a) is more unstable, but as indicated by node 2 in Figure~\ref{fig:ftheta}(a) is easily detected.

\subsubsection{Experiment 2}

In an experiment with ten trials, choosing $\cD(10,32)$, $w = 32$, and $2^{21}$ rectangles, in six trials we computed the Morse graph of Figure~\ref{fig:ML_approx}(b). In the remaining four trials, we again observe the Morse graph of Figure~\ref{fig:ML_approx}(a).
The additional spatial refinement allows for the identification of node 0 with Conley index of a stable period-6 orbit.
In the case of the refined Morse graph the lattice morphism $\nu\colon \sMG \to \sMR_1$ of Theorem~\ref{thm:intro:main} maps nodes 0 and 1 of $\sMG$ to the node 0 of $\sMR_1$.

\subsubsection{Experiment 3}

In an experiment with ten trials, choosing $\cD(20,128)$, $w = 256$, and $2^{23}$ rectangles, we computed the Morse graphs shown in Figure~\ref{fig:ML_approx}(a), (b), (c) and (d), in five, two, one, and two trials, respectively.

As expected the Morse set associated with the unstable fixed point at the origin is still not detected.
This example allows us to highlight that the morphism $\nu$ of Theorem~\ref{thm:intro:main} is dependent on the choice of Morse representation.
In particular, $\nu\colon \sMG \to \sMR_1$ is well defined for all ten trials.
However, for the Morse graph shown in Figure~\ref{fig:ML_approx}(d) we obtain $\nu\colon \sMG \to \sMR_2$ is an embedding. 
It is worth noting that the embedding is obtained even though the data driven computations were done using $2^{23}$ rectangles as opposed to the $2^{27}$ rectangles used to establish $\sMR_2$.

\section{Data and Code Availability}
The data and code used to perform the computations are available at \cite{RCDML}.

\bibliographystyle{siamplain}
\bibliography{references.bib}
\end{document}